\documentclass[11pt,a4paper]{amsart}
\usepackage[utf8]{inputenc}
\usepackage[T1]{fontenc}
\usepackage{mathrsfs}
\usepackage{syntonly}
\usepackage{amsmath}
\usepackage{amsthm}
\usepackage{amsfonts}
\usepackage{amssymb}
\usepackage{latexsym}
\usepackage{url}
\usepackage{amscd,amssymb,amsopn,amsmath,amsthm,graphics,amsfonts,mathrsfs,accents,enumerate,verbatim,calc}
\usepackage{mathtools}
\usepackage[dvips]{graphicx}
\usepackage[colorlinks=true,linkcolor=red,citecolor=blue]{hyperref}
\usepackage[all]{xy}
\usepackage{enumitem}
\usepackage{tikz-cd}
\date{}
\allowdisplaybreaks[4] \footskip=15pt

\numberwithin{equation}{section} \theoremstyle{plain}
\newtheorem{theorem}{Theorem}[section]
\newtheorem{corollary}[theorem]{Corollary}
\newtheorem{lemma}[theorem]{Lemma}
\newtheorem{proposition}[theorem]{Proposition}
\theoremstyle{definition}
\newtheorem{definition}[theorem]{Definition}
\newtheorem{example}[theorem]{Example}
\newtheorem{remark}[theorem]{Remark}

\newtheorem*{ack*}{ACKNOWLEDGEMENTS}

\newcommand{\pf}{\noindent\begin {proof}}
\newcommand{\epf}{\end{proof}}

\newcommand{\Z}{\mathbb{Z}}

\newcommand{\Ext}{{\rm Ext}}

\newcommand{\Hom}{{\rm Hom}}

\newcommand{\Der}{\mathbf{D}}
\newcommand{\mm}{\mathfrak{m}}
\newcommand{\pp}{\mathfrak{p}}
\newcommand{\Spec}{\mathrm{Spec}}
\newcommand{\mSpec}{\mathrm{mSpec}}
\newcommand{\RHom}{\mathbf{R}\Hom}
\newcommand{\ModR}{\mathrm{Mod}\text{-}R}

\newcommand{\ModRm}{\mathrm{Mod}\text{-}R_\mm}
\newcommand{\InjR}{\mathrm{Inj}\text{-}R}
\newcommand{\Ucal}{\mathcal{U}}
\newcommand{\Vcal}{\mathcal{V}}
\newcommand{\Ccal}{\mathcal{C}}
\newcommand{\Fcal}{\mathcal{F}}
\newcommand{\Tcal}{\mathcal{T}}
\newcommand{\Lcal}{\mathcal{L}}
\newcommand{\Hcal}{\mathcal{H}}
\newcommand{\Mcal}{\mathcal{M}}
\newcommand{\Ical}{\mathcal{I}}
\newcommand{\Scal}{\mathcal{S}}
\newcommand{\Ecal}{\mathcal{E}}
\newcommand{\Kcal}{\mathcal{K}}
\newcommand{\Bcal}{\mathcal{B}}
\newcommand{\Ocal}{\mathcal{O}}
\newcommand{\Supp}{\mathrm{Supp}}
\newcommand{\Sub}{\mathrm{Sub}}
\newcommand{\Prod}{\mathrm{Prod}}
\newcommand{\Cogen}{\mathrm{Cogen}}

 \title{Gluing compactly generated t-structures over stalks of affine schemes}

 \author{Michal Hrbek}
 \address[M. Hrbek]{Institute of Mathematics of the Czech Academy of Sciences, \v{Z}itn\'{a} 25, 115 67 Prague, Czech Republic}
 \email{hrbek@math.cas.cz}

 \author{Jiangsheng Hu}
 \address[J. Hu]{School of Mathematics and Physics, Jiangsu University of Technology, Changzhou 213001, China}
 \email{jiangshenghu@jsut.edu.cn}

 \author{Rongmin Zhu} 
 \address[R. Zhu]{Department of Mathematics, Nanjing University, Nanjing 210093, China}
 \email{rongminzhu@hotmail.com}

 \thanks{Michal Hrbek was supported by the GAČR project 20-13778S and RVO: 67985840. Jiangsheng Hu was supported by NSFC grant 11771212. Rongmin Zhu was supported by NSFC grant 11771202.}

\begin{document}
\begin{abstract}
We show that compactly generated t-structures in the derived category of a commutative ring $R$ are in a bijection with certain families of compactly generated t-structures over the local rings $R_\mm$ where $\mm$ runs through the maximal ideals in the Zariski spectrum $\Spec(R)$. The families are precisely those satisfying a gluing condition for the associated sequence of Thomason subsets of $\Spec(R)$. As one application, we show that the compact generation of a homotopically smashing t-structure can be checked locally over localizations at maximal ideals. In combination with a result due to Balmer and Favi, we conclude that the $\otimes$-Telescope Conjecture for a quasi-coherent and quasi-separated scheme is a stalk-local property. Furthermore, we generalize the results of Trlifaj and \c{S}ahinkaya and establish an explicit bijection between cosilting objects of cofinite type over $R$ and compatible families of cosilting objects of cofinite type over all localizations $R_\mm$ at maximal primes.
\end{abstract}
\maketitle

\section {Introduction}
The notion of a t-structure was introduced by Be\u{ı}linson, Bernstein, and Deligne \cite{BBD} in their study of perverse sheaves on an algebraic or analytic variety as a tool for constructing cohomological functors. Later, t-structures turned out to be a natural framework for tilting theory of triangulated categories, see \cite{li16}, \cite{li14} and \cite{ni19}. Such t-structures usually satisfy some kind of finiteness condition, see e.g. \cite{ma18}. The compactly generated t-structures have been studied in depth and in some cases are known to allow for a full classification. For derived categories of commutative noetherian rings, a bijective correspondence between compactly generated t-structures and filtrations of the Zariski spectrum by supports was established by Alonso Tarrío, Jeremías López and Saorín \cite{ajs10}. This was further generalized by the first author to arbitrary commutative rings \cite{hr20} using filtrations by Thomason sets.

Silting theory can be viewed as an adaptation of tilting theory to triangulated categories, and the modern versions of silting theory rely heavily on the notion of a t-structure. Indeed, any cosilting object $C$ is up to equivalence determined by the cosilting t-structure $(^{\perp_{\leq0}}C, {^{\perp_{>0}}C})$. A strong relation between cosilting t-structures and compactly generated t-structures follows from a result by Laking \cite[Theorem 4.6]{la19}. In particular, any compactly generated t-structure is induced by a cosilting object if and only if it is non-degenerate. As a consequence, we call a cosilting object $C$ of \emph{cofinite type} if the t-structure induced by it is compactly generated, see \cite{lihrbek}. Such cosilting objects are abundant, for example, any bounded cosilting complex over a commutative noetherian ring is of cofinite type \cite[Corollary 2.14]{hn20}. Recently, Trlifaj and \c{S}ahinkaya \cite{tr14} constructed a bijective correspondence between (equivalence classes of) n-cotilting modules over a commutative noetherian ring $R$ and (equivalence classes of) compatible families of their colocalizations in all maximal ideals of $R$. This result is the starting point of our inquiry to gluing properties of compactly generated t-structures. The notion of colocalization in $\ModR$ is due to Melkerson and Schenzel (cf. \cite[p.118]{xu1996}) with similar constructions already used in \cite{ca1956}, and a derived version of colocalization will play an essential role in our approach as well.

A Bousfield localization of a triangulated category is called smashing if it commutes with all coproducts. The Telescope Conjecture (TC) originates from the work of Ravenel in algebraic topology \cite{Ra1984} and asks whether any such smashing localization is generated by compact objects. It is a landmark result of Neeman \cite{ne1992} that (TC) holds in the derived category $\Der(R)$ of a commutative noetherian ring $R$. On the other hand, Keller \cite{ke1994} established examples of non-noetherian commutative rings for which (TC) fails. For non-stable t-structures, a generalization of the smashing property was introduced by Saorín, Šťovíček and Virili \cite{ssv17}. This class of homotopically smashing t-structures encompasses both the smashing Bousfield localizations and the t-structures induced by pure-injective cosilting objects, this follows from the work of Krause \cite{kr2000} and Laking \cite{la18}. Recently, the first author and Nakamura showed in \cite{hn20} that any homotopically smashing t-structure in the derived category of a commutative noetherian ring is compactly generated,
which generalizes the validity of (TC) for commutative noetherian rings.
As a consequence, we say that the derived category $\Der(R)$ of a (not necessarily noetherian) ring $R$ satisfies the Semistable Telescope Conjecture (STC) if any homotopically smashing t-structure is compactly generated.

The aim of this paper is to glue compactly generated t-structures and cosilting objects of cofinite type over all (co)localizations at maximal ideals, and to study the stalk-local properties of the (Semistable) Telescope Conjecture. For this purpose, we use the description of these t-structures of \cite{ajs10} and \cite{hr20} by geometric invariants and first introduce the ``compatibility'' condition for the family $\{\mathbb{X}(\mm) \mid \mm \in \mSpec(R)\}$ of Thomason filtrations, which we then demonstrate to correspond precisely to the case in which this collection naturally glues over the cover of $\Spec(R)$ by the subsets homeomorphic to $\Spec(R_\mm)$ (see Definition \ref{def:compatible-of-Thomason} and Proposition \ref{prop_corr_thom}). We remark that our gluing condition Definition \ref{def:compatible-of-Thomason} is nothing but a suitable generalization of the condition used in \cite{tr14} to the setting of rings which are not necessarily noetherian, see Remark~\ref{rem_noeth}.

Let us briefly list the highlights of the present paper, which are all obtained using the gluing technique described above.

\begin{enumerate}
    \item[(1)] In Theorem \ref{thm-compact-t-structure} we glue (non-degenerate) compactly generated t-structures over all localizations at maximal ideals which is based on the gluing of the corresponding Thomason filtrations via Proposition \ref{prop_corr_thom}. More  specifically, it is proved that there is a bijective correspondence between (non-degenerate) compactly generated t-structures $(\Ucal,\Vcal)$ in $\Der(R)$ and compatible families $\{(\Ucal(\mm),\Vcal(\mm)) \mid \mm \in \mSpec(R)\}$ of (non-degenerate) compactly generated t-structures.

\item[(2)] In Theorem \ref{thm_tc_loc} we obtain a stalk-local criterion for (Semistable) Telescope Conjecture by applying the local-global property of compact generation established in Proposition \ref{prop_cg_loc}. More precisely, for any commutative ring $R$, it is proved that (Semistable) Telescope Conjecture holds in $\Der(R)$ if and only if (Semistable) Telescope conjecture holds in $\Der(R_\mm)$ for any maximal ideal $\mm$ of $R$. One corollary is that both (TC) and (STC) hold for any commutative ring $R$ all of which stalks $R_\mm$ are noetherian. Examples of such rings include non-noetherian rings like von Neumann regular rings, recovering a result of Bazzoni-Šťovíček, see \cite[\S 7]{bs17}. Our result also has consequences for non-affine schemes. Indeed, in combination with the result of Balmer and Favi we obtain that the $\otimes$-Telescope Conjecture ($\otimes$TC) for a quasi-coherent and quasi-separated scheme is not just affine-local, but even a stalk-local property, see \S\ref{ss_stalk}.

\item[(3)] We establish a bijective
correspondence between cosilting objects in $\Der(R)$ of cofinite type up to equivalence and compatible families $\{C(\mm) \mid \mm \in \mSpec(R)\}$ of cosilting objects of cofinite type up to equivalence (see Theorem \ref{thm4.17}). We give applications to pure-injective cosilting objects, $n$-term cosilting objects, cotilting modules, and cosilting modules over commutative noetherian rings (see Corollaries \ref{cor_result_noeth}, \ref{cor:cotilting-correspond} and \ref{corollary:cosilting-correspond}). It should be noted that our correspondence here restricts to one of \cite[Corollary 3.6]{tr14}, and the notions of equivalence and compatible condition on families here restrict perfectly well to the notions used in \cite{tr14}. In Section~\ref{S:silting}, we obtain a similar result for silting objects under slightly stronger assumptions. Unlike in the cosilting setting however, the gluing of local silting object to form a global one is not explicit, see Remark~\ref{rem_silting_coproduct}.
\end{enumerate}

\textbf{Notation.} All subcategories are always considered to be full and closed under isomorphisms. Complexes are written using the cohomological notation, meaning that the degree increases in the direction of differential maps. The $n$-th power of the suspension functor of any triangulated category will be denoted as $[n]$ for $n \in \Z$.

\section{Preliminaries}\label{S:prelim}
In this section, $\Tcal$ will always denote a compactly generated triangulated category underlying a Grothendieck derivator. We refer to reader to \cite[Appendix]{hn20} for a source of references and basic terminology about such categories well-suited for our objectives. In particular, this assumption implies that $\Tcal$ has all set-indexed products and coproducts. The enrichment of the Grothendieck derivator allows for computing of homotopy limits and colimits. In fact, we will be mostly interested in the case when $\Tcal = \Der(R)$ is the unbounded derived category of the module category $\ModR$ of a  (not necessarily noetherian) commutative ring $R$.
Also, the only homotopy construction we will be interested in computing is that of a \emph{directed homotopy colimits}. In the case of $\Der(R)$, the \emph{directed homotopy colimits} are precisely the direct limits (= directed colimits) constructed in $\mathbf{C}(R)$, the Grothendieck category of cochain complexes of $R$-modules. In particular, a subcategory $\Ccal$ of $\Der(R)$ is \emph{closed under directed homotopy colimits} if it is closed under direct limits computed in $\mathbf{C}(R)$.

Recall that an object $S \in \Tcal$ is \emph{compact} if the covariant functor $\Hom_{\Tcal}(S,-): \Der(R) \rightarrow \mathrm{Ab}$ preserves coproducts. The symbol $\Tcal^c$ will denote the subcategory of $\Tcal$ consisting of all compact objects of $\Tcal$. Recall that an object $S \in \Der(R)$ belongs to $\Der(R)^c$ if and only if $S$ is isomorphic in $\Der(R)$ to a bounded complex of finitely generated projective $R$-modules.

\subsection{t-structures}
 A \emph{t-structure} is a pair $(\Ucal,\Vcal)$ of full subcategories of $\Tcal$ which satisfy the following axioms:
\begin{enumerate}
    \item[(t1)] $\Hom_{\Tcal}(\Ucal,\Vcal) = 0$,
    \item[(t2)] $\Ucal[1] \subseteq \Ucal$ (and $\Vcal[-1] \subseteq \Vcal$),
    \item[(t3)] for any $X \in \Tcal$ there is a triangle
            $$U \rightarrow X \rightarrow V \rightarrow U[1]$$
            with $U \in \Ucal$ and $V \in \Vcal$.
\end{enumerate}

The class $\Ucal$ is called the \emph{aisle} and $\Vcal$ is called the \emph{coaisle} of the t-structure. We recall the well-known fact that the triangle in axiom (t3) is uniquely determined and functorial --- indeed, it follows from the axioms that $\Ucal$ is a coreflective subcategory of $\Tcal$ and the map $U \rightarrow X$ is precisely the $\Ucal$-coreflection of $X$, see \cite{kv1988}. The dual statement is valid for the coaisle $\Vcal$ and the map $X \rightarrow V$ as well --- in particular, $\Vcal$ is a reflective subcategory of $\Tcal$.

A t-structure $(\Ucal,\Vcal)$ is called:
\begin{itemize}
    \item \emph{stable} if $\Ucal[-1] \subseteq \Ucal$ (or equivalently, $\Vcal[1] \subseteq \Vcal$);
    \item \emph{non-degenerate} if $\bigcap_{n \in \Z}\Ucal[n] = 0$ and $\bigcap_{n \in \Z}\Vcal[n] = 0$.
\end{itemize}

The aisles of stable t-structures are precisely the kernels of Bousfield localization functors. The non-degeneracy condition holds precisely when the cohomological functor induced by the t-structure detects zero object. Clearly, the two conditions are mutually exclusive whenever $\Tcal$ contains non-zero objects.
\subsection{Purity and definable subcategories in triangulated setting}
We briefly recall the theory of purity in $\Tcal$, first introduced by Beligiannis \cite{be2000} and Krause \cite{kr2000}. We call a map $f$ in $\Tcal$ a \emph{pure monomorphism} (resp. \emph{pure epimorphism}) provided that $\Hom_{\Tcal}(S,f)$ is a monomorphism (resp. an epimorphism)  for any $S \in \Tcal^c$. Note that, in the triangulated world, a pure monomorphism does not have to be a categorical monomorphism. An object $E \in \Tcal$ is called \emph{pure-injective} provided that any pure monomorphism starting in $E$ is a split monomorphism in $\Tcal$. Let $\Ccal$ be a full subcategory of $\Tcal$. We say that $\Ccal$ is closed under pure monomorphisms if for any pure monomorphism $f: X \rightarrow Y$ with $Y \in \Ccal$ we also have $X \in \Ccal$ and we define the analogous notion of subcategory closed under pure epimorphisms similarly.

A subcategory $\Ccal$ is \emph{definable} provided that there is a set $\Phi$ of maps between objects of $\Tcal^c$ such that $\Ccal = \{X \in \Tcal \mid \Hom_R(f,X) \text{ is surjective for any $f \in \Phi$}\}$. Similarly to their more classical counterparts in module categories, definable subcategories of $\Tcal$ can be characterized by their closure properties. Indeed, we have the following result due to Laking and Laking-Vit\'{o}ria.
\begin{theorem}\emph{(\cite[Theorem 3.11]{la18}, \cite[Theorem 4.7]{la19})}
    The following are equivalent for a subcategory $\Ccal$ of $\Tcal$:
    \begin{enumerate}
        \item[(i)] $\Ccal$ is definable,
        \item[(ii)] $\Ccal$ is closed under products, pure monomorphisms, and pure epimorphisms.
        \item[(iii)] $\Ccal$ is closed under products, pure monomorphisms, and directed homotopy colimits.
    \end{enumerate}
\end{theorem}
\subsection{Homotopically smashing t-structures}
A t-structure $(\Ucal,\Vcal)$ is called \emph{smashing} if the coaisle $\Vcal$ is closed under coproducts. The aisles of stable smashing t-structures are precisely the kernels of \emph{smashing localization functors} of $\Tcal$, see \cite[A.5]{hn20}. For t-structures which are non-stable, a stronger condition is often needed, here we follow \cite{ssv17}. A t-structure $(\Ucal,\Vcal)$ is called \emph{homotopically smashing} if $\Vcal$ is closed under directed homotopy colimits. For stable t-structures, this is equivalent to the smashing property \cite{kr2000}. A priori, this is a weaker condition than requiring the coaisle $\Vcal$ to be a definable subcategory. However, a recent result due to Saorín and Šťovíček \cite{ss20} shows that at least in the algebraic setting (in particular, in the case $\Tcal = \Der(R)$), these two conditions coincide. Furthermore, Angeleri-H\"{u}gel, Marks, and Vitória \cite{li17} showed that coaisles of such t-structures are fully determined by their closure properties.

Recall that a subcategory $\Ccal$ of $\Tcal$ is \emph{suspended} (respectively, \emph{cosuspended}) if $\Ccal$ is closed under extensions and $\Ccal[1]\subseteq \Ccal$ (respectively, $\Ccal[-1]\subseteq \Ccal$). Then we can summarize the two above mentioned results about homotopically smashing t-structures in derived categories.
\begin{theorem}\emph{\cite{ss20,li17}}
The following conditions are equivalent for a subcategory $\Vcal$ of $\Der(R)$:
\begin{enumerate}
    \item [(i)] $\Vcal$ is the coaisle of a homotopically smashing t-structure $(\Ucal,\Vcal)$,
    \item [(ii)] $\Vcal$ is definable and cosuspended.
\end{enumerate}
\end{theorem}
\begin{proof}
    Any coaisle is clearly a cosuspended subcategory. If $(\Ucal,\Vcal)$ is homotopically smashing t-structure then $\Vcal$ is definable by \cite[Remark 8.9]{ss20}. On the other hand, if $\Vcal$ is a definable and cosuspended subcategory then it is a coaisle of a t-structure by \cite[Lemma 4.8]{li17}, and such t-structure is clearly homotopically smashing.
\end{proof}
\subsection{Orthogonal subcategories}
Let $\Ccal$ be a subcategory of $\Tcal$. We write $\Hom_{\Tcal}(\Ccal,X) = 0$ as a shorthand for the statement $\Hom_{\Tcal}(C,X) = 0$ for all $C \in \Ccal$. Given a subset $I$ of $\Z$, we also use the following notation for subcategories degreewise orthogonal to $\Ccal$:
$$\Ccal^{\perp_I} = \{X \in \Tcal \mid \Hom_{\Tcal}(\Ccal,X[i]) = 0 ~\forall i \in I\}.$$
The role of $I$ will be played by the symbols $0$, $\leq 0$, $<0$, $\geq 0$, $>0$, or $\Z$ with their obvious interpretations as subsets of $\Z$.
The symbols $\Hom_{\Tcal}(X,\Ccal) = 0$ and ${}^{\perp_I}\Ccal$ are defined analogously, in particular:
$${}^{\perp_I}\Ccal = \{X \in \Tcal \mid \Hom_{\Tcal}(X,\Ccal[i]) = 0 ~\forall i \in I\}.$$
If $\Ccal = \{C\}$ is a singleton for some object $C \in \Tcal$, we will omit the brackets and write just $C^{\perp_0}$ et cetera.
\subsection{Compactly generated t-structures}
We say that a t-structure $(\Ucal,\Vcal)$ in $\Tcal$ is \emph{compactly generated} if there is a set $\Scal \subseteq \Tcal^c$ of compact objects such that $\Vcal = \Scal^{\perp_0}$. Any compactly generated t-structure is homotopically smashing, see \cite[Proposition 5.4]{ssv17}.

The compactly generated t-structures in $\Der(R)$ admit a geometrical classification in terms of invariants coming from the dual topology on $\Spec(R)$, which we recall now. A subset $X$ of $\Spec(R)$ is called \emph{Thomason} provided that there is a set $\Ical$ of finitely generated ideals of $R$ such that $X = \bigcup_{I \in \Ical}V(I)$. We remark that a subset of the spectrum is Thomason precisely if it is an open subset with respect to the Hochster dual topology on $\Spec(R)$, as explained in the discussion \cite[\S 2]{hs20}. Also note that if $R$ is noetherian, Thomason subsets are precisely the \emph{specialization closed} subsets of $\Spec(R)$, that is, the upper subsets of the poset $(\Spec(R), \subseteq)$. A \emph{Thomason filtration} is a sequence $\mathbb{X} = (X_n \mid n \in \Z)$ of Thomason subsets of $\Spec(R)$ which is decreasing in the sense that $X_n \supseteq X_{n+1}$ for each $n \in \Z$.

It turns out that any compactly generated t-structure in $\Der(R)$ is generated by distinguished compact objects of $\Der(R)$, the Koszul complexes. Recall that if $x \in R$ is an element then the \emph{Koszul complex} of $x$ is the complex of the form $K(x)= (R \xrightarrow{\cdot x} R)$ concentrated in degrees -1 and 0. If $\bar{x} = (x_1,x_2,\ldots,x_n)$ is a finite sequence of elements then we define the Koszul complex $K(\bar{x}) = \bigotimes_{i=1}^n K(x_i)$. Let $I$ be a finitely generated ideal of $R$. Then we will abuse the notation and write $K(I)$ for the Koszul complex on any fixed finite sequence $\bar{x}$ of generators of $I$. Recall that $K(I)$ is always a compact object in $\Der(R)$ and also that its cohomology modules are supported on $V(I)$. Although the change of choice of generators may alter $K(I)$ even when considered as an object of $\Der(R)$ up to isomorphism, these complexes generate the same t-structure regardless of the choice of generators, and therefore the abuse in notation is harmless in our application. The following results generalizes the classification in the noetherian case considered in \cite{ajs10}.
\begin{theorem}\label{thmcompgen}\emph{(\cite[Theorem 5.1]{hr20})}
            Let $R$ be a commutative ring. There is a bijective correspondence between the following collections:

            (i) compactly generated t-structures $(\Ucal,\Vcal)$ in $\Der(R)$, and

            (ii) Thomason filtrations $\mathbb{X} = (X_n \mid n \in \Z)$ on $\Spec(R)$.

            This correspondence assigns to a Thomason filtration $\mathbb{X}$ the coaisle of the form $\Vcal = \{K(I)[-n] \mid V(I) \subseteq X_n, n \in \Z\}^{\perp_0}$.
\end{theorem}
\subsection{Hereditary torsion pairs in $\ModR$}\label{SS:torsion}
Recall that a \emph{torsion pair} in $\ModR$ is a pair of full subcategories $(\Tcal,\Fcal)$ of $\ModR$ which are maximal with respect to the property $\Hom_R(\Tcal,\Fcal) = 0$. A torsion pair $(\Tcal,\Fcal)$ is called \emph{hereditary} if the \emph{torsion class} $\Tcal$ is closed under submodules, or equivalently, if the \emph{torsion-free class} $\Fcal$ is closed under taking injective envelopes. A hereditary torsion pair $(\Tcal,\Fcal)$ is \emph{of finite type} if $\Fcal$ is closed under direct limits. Hereditary torsion pairs of finite type in $\ModR$ were proved by Garkusha and Prest to be in bijection with Thomason subsets of $\Spec(R)$.
\begin{theorem}\label{thm_torsion_pair}\emph{(\cite{gp08}, see also \cite[Proposition 2.11]{hs20})}
    There is a bijection
    $$\left \{ \begin{tabular}{ccc} \text{ Hereditary torsion pairs $(\Tcal,\Fcal)$} \\ \text{in $\ModR$} \end{tabular}\right \}  \xleftrightarrow{1-1}  \left \{ \begin{tabular}{ccc} \text{ Thomason subsets $X$} \\ \text{ of $\Spec(R)$ } \end{tabular}\right \}$$
    provided by the mutually inverse assignments
    $$\Tcal \mapsto X = \bigcup\{V(I) \mid R/I \in \Tcal\}$$
    and
    $$X \mapsto \Tcal = \{M \in \ModR \mid \Supp(M) \subseteq X\}.$$
\end{theorem}

\section {(Co)localization and t-structures }\label{sec_three}

\begin{definition}For every object $X\in \Der(R)$ and every prime ideal $\pp$ of $R$, we denote by $X_{\pp}$ the object $X \otimes_R^\mathbf{L} R_\pp = X \otimes_R R_\pp$, the localization of $X$ at $\pp$ and by $X^\pp$ the object $\mathbf{R}\mathrm{Hom}_{R}(R_{\pp},X)$; we call it the \emph{colocalization} of $X$ at $\pp$.
Similarly, for a subcategory $\Ccal$ of $\Der(R)$ we consider the following subcategories of $\Der(R_\pp)$: $\Ccal_\pp = \{X_\pp \mid X \in \Ccal\}$ and $\Ccal^{\pp}=\{X^\pp \mid X\in\Ccal\}$.
\end{definition}
\begin{remark}\label{rem_adjoints}
Recall that for any prime ideal $\pp$, the derived category $\Der(R_\pp)$ is naturally a (full) subcategory of $\Der(R)$. By the same token, we can naturally consider any subcategory $\Ccal_\pp$ of $\Der(R_\pp)$ as a subcategory of $\Der(R)$. Moreover, there is a \emph{(stable) TTF-triple} $(\Lcal,\Der(R_\pp),\Kcal)$ in $\Der(R)$, which amounts to saying that there are two adjacent t-structures $(\Lcal,\Der(R_\pp))$ and $(\Der(R_\pp),\Kcal)$ (both of which are necessarily stable). It follows that the inclusion of $\Der(R_\pp)$ into $\Der(R)$ admits both the left and the right adjoint, and these are realized by the functors $- \otimes_R^\mathbf{L} R_\pp$ and $\RHom_R(R_\pp,-)$, respectively. For details, see e.g. \cite[Theorem 4.3]{lihrbek} and references therein.
\end{remark}
\begin{lemma}\label{lem_adjunction}Let $\pp$ be a prime ideal of $R$. We have the following
natural isomorphisms:

$(i)$ For every  $X\in \Der(R_{\pp})$ and $Y\in\Der(R)$, we have
$$\Hom_{\Der(R_{\pp})}(X,Y^{\pp})\cong \Hom_{\Der(R)}(X,Y).$$

$(ii)$  For every  $Y\in \Der(R_{\pp})$ and  $X\in \Der(R)$, we have
$$\Hom_{\Der(R_{\pp})}(X_{\pp},Y)\cong \Hom_{\Der(R)}(X,Y).$$

$(iii)$ For  $X, Y\in \Der(R)$, we have
$$\Hom_{\Der(R)}(X_{\pp},Y)\cong \Hom_{\Der(R)}(X,Y^{\pp}).$$

\end{lemma}
\begin{proof}
 For any object $X\in \Der(R_\pp)$ and maximal ideal $\pp$, we have $X_{\pp}=X\otimes_{R}R_{\pp}\cong X\otimes^{\mathbf{L}}_{R}{R_{\pp}}\cong X\otimes^{\mathbf{L}}_{R_{\pp}}{R_{\pp}}\cong X$ and $X^{\pp}=\RHom_{R}(R_{\pp},X)\cong\RHom_{R_{\pp}}(R_{\pp},X)\cong X$.
The isomorphism in $(i)$ follows by $\Hom_{\Der(R_{\pp})}(X,Y^{\pp})\cong
\Hom_{\Der(R_{\pp})}(X,\RHom_{R}(R_{\pp},Y))
\cong \Hom_{\Der(R)}(X\otimes^{\mathbf{L}}_{R}R_{\pp},Y)\cong \Hom_{\Der(R)}(X,Y)$. Similarly, the isomorphism in $(ii)$ follows by $\Hom_{\Der(R_{\pp})}(X_{\pp},Y)\cong \Hom_{\Der(R)}(X,\RHom_{R}(R_{\pp},Y))\cong\Hom_{\Der(R)}(X,Y)$.
Applying the isomorphisms from $(i)$ and $(ii)$, we obtain $\Hom_{\Der(R)}(X_{\pp},Y)\cong \Hom_{\Der(R_{\pp})}(X_{\pp},Y^{\pp})\cong \Hom_{\Der(R)}(X,Y^{\pp})$.
\end{proof}
\begin{lemma}\label{lemma_tstr_loc}
    Let $(\Ucal,\Vcal)$ be a t-structure in $\Der(R)$ and $\pp \in \Spec(R)$. Then:
    \begin{enumerate}
        \item[(i)] $\Ucal_\pp = \Ucal \cap \Der(R_\pp)$ and $\Vcal^\pp = \Vcal \cap \Der(R_\pp)$.
        \item[(ii)] $(\Ucal_\pp,\Vcal^\pp)$ is a t-structure in $\Der(R_\pp)$.
        \item[(iii)] If $(\Ucal,\Vcal)$ is in addition homotopically smashing then $\Vcal^\pp = \Vcal_\pp$.
    \end{enumerate}
\end{lemma}
\begin{proof}
    $(i):$ By \cite[Proposition 2.2]{hr20}, we have that $U_\pp \in \Ucal$ for any $U \in \Ucal$ and $V^\pp \in \Vcal$ for any $V \in \Vcal$, which easily yields the desired equalities.

    $(ii):$ The only non-trivial step is to check the existence of canonical triangles with respect to $(\Ucal_\pp,\Vcal^\pp)$. Let $X$ be an object of $\Der(R_\pp) \subseteq \Der(R)$ and consider the canonical triangle
    $$U \rightarrow X \xrightarrow{r} V \rightarrow U[1]$$
    with respect to the t-structure $(\Ucal,\Vcal)$ in $\Der(R)$. Consider the natural colocalization morphism $c: V^\pp \rightarrow V$. Recall from Remark~\ref{rem_adjoints} that $c$ is the $\Der(R_\pp)$-coreflection of $V$. Together with $X \in \Der(R_\pp)$ this yields that the morphism $r$ factors through $c$ by a map $f: X \rightarrow V^\pp$. On the other hand, as $V^\pp$ belongs to $\Vcal$ by $(i)$, the map $f$ factors through the $\Vcal$-reflection map $r$ by a map $g: V \rightarrow V^\pp$. The situation is captured in the following commutative diagram:
%\[
   {\begin{center}
    \begin{tikzcd}[column sep = large, row sep = large]
        X \arrow[r,"r"] \arrow[dr,"f"] & V \arrow[d, bend left=20,"g"] \\
        & V^\pp \arrow[u,bend left=20,"c"]
    \end{tikzcd}
    \end{center}}
%\]
    By the construction, we have $cgr = cf = r$. Since $r$ is the $\Vcal$-reflection morphism of $X$, the map $cg$ has to be the identity on $V$, and so $g$ is a split monomorphism. It follows that $V \in \Der(R_\pp)$, which implies also $U \in \Der(R_\pp)$. We conclude that $U \rightarrow X \xrightarrow{r} V \rightarrow U[1]$ is already the desired approximation triangle in $\Der(R_\pp)$.

    $(iii):$ If $(\Ucal,\Vcal)$ is homotopically smashing then $\Vcal$ is closed under the localization functor $- \otimes_R R_\pp$, and therefore $\Vcal_\pp = \Vcal \cap \Der(R_\pp)$. By $(i)$, $\Vcal_\pp = \Vcal^\pp$.
\end{proof}
\begin{lemma}\label{lem_def_loc}
    Let $\Vcal$ be a definable subcategory of $\Der(R)$. Then the following conditions are equivalent for any object $X \in \Der(R)$:
    \begin{enumerate}
        \item[(i)] $X \in \Vcal$,
        \item[(ii)] $X_\mm \in \Vcal_\mm$ for any $\mm \in \mSpec(R)$,
        \item[(iii)] $X_\mm \in \Vcal$ for any $\mm \in \mSpec(R)$.
    \end{enumerate}
\end{lemma}
\begin{proof}
    $(i) \implies (ii):$ This is just the definition of the subcategory $\Vcal_\mm$.

    $(ii) \implies (iii):$ Since $\Vcal$ is definable, it is closed under directed homotopy colimits. But since $X_\mm = X \otimes_R R_\mm$, $X_\mm$ can be represented as a directed homotopy colimit of a coherent diagram consisting of finite coproducts of copies of $X$. Therefore, $\Vcal_\mm$ is a subcategory of $\Vcal$.

    $(iii) \implies (i):$ Consider the natural map $f: X \rightarrow \prod_{\mm \in \mSpec(R)} X_\mm$. If $S \in \Der(R)^c$ is a compact object, we have that $\Hom_{\Der(R)}(S,f): \Hom_{\Der(R)}(S,X) \rightarrow \Hom_{\Der(R)}(S,\prod_{\mm \in \mSpec(R)}X_\mm) \cong \prod_{\mm \in \mSpec(R)}\Hom_{\Der(R)}(S,X)_\mm$ is a monomorphism in $\ModR$, and therefore $f$ is a pure monomorphism in $\Der(R)$. Since $\Vcal$ contains $X_\mm$ for any $\mm \in \mSpec(R)$ and is closed under products and pure monomorphisms, we conclude that $X \in \Vcal$.
\end{proof}
\subsection{Aisles of compactly generated t-structures}
If $R$ is noetherian, a more explicit description of the compactly generated t-structure $(\Ucal,\Vcal)$ in terms of the associated Thomason filtration is available in \cite[Theorem 3.11]{ajs10}. In this case, the aisle is described cohomologically as $\Ucal = \{X \in \Der(R) \mid \Supp H^n(X) \subseteq X_n ~\forall n \in \Z\}$. For general commutative ring and general Thomason filtration, such description seems currently unavailable apart from special cases \cite[Proposition 6.6]{ro08}, \cite[Lemma 3.10]{hs20}. However, we can extract the following slightly weaker structural information even in the general situation which we record here for later use.

Let $\mathbb{X} = (X_n \mid n \in \Z)$ be a Thomason filtration of $\Spec(R)$ inducing a compactly generated t-structure $(\Ucal,\Vcal)$. Let $\Der^+(R)$ denote the subcategory of $\Der(R)$ consisting of all objects $X$ with $H^i(X) = 0$ for $i \ll 0$.

\begin{lemma}\label{lem_bounded_dev}
    Let $(\Ucal,\Vcal)$ be a compactly generated t-structure in $\Der(R)$ corresponding to a Thomason filtration $\mathbb{X} = (X_n \mid n \in \Z)$. Put $\Ucal_\# = \{X \in \Der(R) \mid \Supp H^n(X) \subseteq X_n ~\forall n \in \Z\}$. Then:
    \begin{enumerate}
        \item[(i)] There is a set $\Ecal$ of shifts of stalks of injective $R$-modules such that $\Ucal_\# = {}^{\perp_0}\Ecal$.
        \item[(ii)] The category $\Ucal_\#$ is an aisle of a t-structure $(\Ucal_\#,\Vcal_\#)$.
        \item[(iii)] $\Ucal \subseteq \Ucal_\#$ and $\Vcal_\# \subseteq \Vcal$,
        \item[(iv)] $\Ucal \cap \Der^+(R) = \Ucal_\# \cap \Der^+(R)$ and $\Vcal \cap \Der^+(R) = \Vcal_\# \cap \Der^+(R)$.
    \end{enumerate}
\end{lemma}
\begin{proof}
    $(i)$: Recall from \S\ref{SS:torsion} that $\Tcal_n = \{M \in \ModR \mid \Supp(M) \subseteq X_n\}$ is a torsion class of a hereditary torsion pair $(\Tcal_n,\Fcal_n)$, which amounts to saying that there is a set $\Ecal_n$ of injective $R$-modules such that $\Tcal_n = \{M \in \ModR \mid \Hom_R(M,E) = 0 ~\forall E \in \Ecal_n\}$. Then  $\Ucal_\# = {}^{\perp_0}\Ecal$, where $\Ecal = \bigcup_{n \in \Z}\Ecal_n[-n]$ (see \cite[Lemma 3.2]{hr20}).

    $(ii)$: This follows from $(i)$, \cite[Corollary 5.4]{la19}, and the easy observation that $\Ucal_\#$ is a suspended subcategory of $\Der(R)$.

    $(iii)$: Since $\Vcal = \Ucal^{\perp_0}$ and $\Vcal_\# = \Ucal_\#^{\perp_0}$, it is clearly enough to show the first claim. Recall from \cite[\S 1.2]{ajs10} that $\Ucal$ is the smallest suspended subcategory of $\Ucal$ containing all objects $K(I)[-n]$ where $V(I) \subseteq X_n, n \in \Z$ closed under coproducts. Since $\Supp H^n(K(I)) \subseteq V(I)$ for any finitely generated ideal $I$ and any $n \in \Z$, we see that all the objects $K(I)[-n]$ belong to $\Ucal_\#$. Finally, since $\Ucal_\#$ is an aisle, we conclude that $\Ucal \subseteq \Ucal'$.

    $(iv)$: By $(iii)$ we already have inclusions $\Ucal \cap \Der^+(R) \subseteq \Ucal_\# \cap \Der^+(R)$ and $\Vcal \cap \Der^+(R) \supseteq \Vcal_\# \cap \Der^+(R)$. Let $X \in \Ucal_\# \cap \Der^+(R)$. Since the subcategory $\Ucal_\#$ is defined by cohomology, we clearly have that any soft (cohomological) truncation $\tau^{<n}(X)$ belongs to $\Ucal_\#$. Also, the aisle $\Ucal$ is closed under extensions and homotopy colimits \cite[Proposition 4.2]{ssv17}. Since $X$ can be expressed as a directed homotopy colimit of its soft truncations, $X \cong \mathrm{hocolim}_{n >0}\tau^{<n}(X)$, it is enough to show that $\tau^{<n}(X) \in \Ucal$ for each $n > 0$. Since $X \in \Der^+(R)$, $\tau^{<n}(X)$ is obtained by a finitely iterated extension of the cohomology stalks $H^k(X)[-k]$. Therefore, it is enough to consider the case of a stalk complex $X = M[-k]$ for an $R$-module $M$ satisfying $\Supp(M) \subseteq X_k$. The latter condition implies that any annihilator $I$ of an element of $M$ satisfies $V(I) \subseteq X_k$ (\cite[Proposition 2.11]{hs20}). Therefore, $M$ can be constructed from cyclic modules $R/I$ satisfying $V(I) \subseteq X_k$ using extensions and direct limits, and so $M[-k] \in \Ucal$ by \cite[Lemma 5.3]{hr20}.
\end{proof}
For any $R$-module $M$, let us denote by $E(M)$ the injective envelope of $M$ in $\ModR$.
\begin{lemma}\label{lem_tstr_techn}
    Let $(\Ucal,\Vcal)$ be a compactly generated t-structure in $\Der(R)$ corresponding to a Thomason filtration $\mathbb{X} = (X_n \mid n \in \Z)$. Then the following statements hold true:
    \begin{enumerate}
        \item[(i)] For any $\pp \in \Spec(R)$, $\kappa(\pp)[-n] \in \Ucal$ if and only if $\pp \in X_n$.
        \item[(ii)] For any $X \in \Vcal \cap \Der^+(R)$ we have $E(H^{\inf(X)}(X))[-\inf(X)] \in \Vcal$, where $\inf(X) = \min\{k \in \Z \mid H^k(X) \neq 0\}$,
        \item[(iii)] Let $\Ecal_n$ be the subcategory of $\ModR$ consisting of all injective $R$-modules $E$ such that $E[-n] \in \Vcal$. Then for any $\pp \in \Spec(R)$ we have $\Hom_R(\kappa(\pp),\Ecal_n) = 0$ if and only if $\pp \in X_n$.
    \end{enumerate}
\end{lemma}
\begin{proof}
    $(i)$: By Lemma~\ref{lem_bounded_dev}(iv), $\kappa(\pp)[-n] \in \Ucal$ if and only if $\kappa(\pp)[-n] \in \Ucal_\#$, or equivalently, $\Supp (\kappa(\pp)) = \{\pp\} \subseteq X_n$.

    $(ii)$: This follows from \cite[Lemma 3.3]{hr20}.

    $(iii)$: If $E[-n]$ belongs to $\Vcal$ then $\Hom_R(\kappa(\pp),E) = \Hom_{\Der(R)}(\kappa(\pp)[-n],E[-n]) = 0$ for any $\pp \in X_n$ because $\kappa(\pp)[-n] \in \Ucal$ by $(i)$. For the converse, assume that $\pp \not\in X_n$ and let $X \in \Ucal$. Since $E(\kappa(\pp))$ is injective, we have $\Hom_{\Der(R)}(X,E(\kappa(\pp))[-n]) = \Hom_R(H^n(X),E(\kappa(\pp)))$ by \cite[Lemma 3.2]{hr20}. By Lemma~\ref{lem_bounded_dev}(iii), $\Supp H^n(X) \subseteq X_n$, and so $\Hom_R(H^n(X),E(\kappa(\pp))) = 0$. This shows that $E(\kappa(\pp))[-n] \in \Ucal^{\perp_0} = \Vcal$, and therefore $E(\kappa(\pp)) \in \Ecal_n$. Then clearly $\Hom_R(\kappa(\pp),\Ecal_n) \neq 0$.
\end{proof}
\subsection{Compatible families of Thomason sets}\label{ss_comp_thom}
Our next step is to show how Thomason sets can be glued together from local data over all localizations at maximal ideals. Let $X(\mm)$ be a Thomason subset of $\Spec(R_\mm)$ for each $\mm \in \mSpec(R)$. If $Y$ is any subset of $\Spec(R_\mm)$, we will denote by $Y^*$ its image under the natural inclusion $\Spec(R_\mm) \xhookrightarrow{} \Spec(R)$ induced by the localization map (the choice of maximal ideal $\mm$ will always be clear from the context). Note that $Y = \{\pp_\mm \mid \pp \in Y^*\} \subseteq \Spec(R_\mm)$. We consider the following condition for the family $\{X(\mm) \mid \mm \in \mSpec(R)\}$:
\begin{equation}\label{gluingcondition}\tag{$\dagger$}
    \forall \mm,\mathfrak{m'} \in \mSpec(R): \{\pp \in X(\mm)^* \mid \pp \subseteq \mathfrak{m'}\} = \{\pp \in X(\mathfrak{m'})^* \mid \pp \subseteq \mm\}.
\end{equation}
\begin{lemma}\label{lem4.9}
    In the setting as above, put $X = \bigcup_{\mm \in \mSpec(R)}X(\mm)^*$. Then the following conditions are equivalent:
    \begin{enumerate}
        \item[(i)] The set $\{X(\mm) \mid \mm \in \mSpec(R)\}$ satisfies condition (\ref{gluingcondition}) and $X$ is a Thomason subset of $\Spec(R)$,
        \item[(ii)] $X = \bigcup_{I \in \Ical} V(I)$ where $\Ical$ is the set of all finitely generated ideals of $R$ such that $V(I_\mm) \subseteq X(\mm)$ for all $\mm \in \mathrm{Spec(R)}$.
    \end{enumerate}
\end{lemma}
\begin{proof}
    $(i) \implies (ii):$ Set $X' =  \bigcup_{I \in \Ical} V(I)$ and let us show $X = X'$. Fix $I \in \Ical$, $\pp \in V(I)$, and $\mm \in \mSpec(R)$. Since $V(I_\mm) \subseteq X(\mm)$, we clearly have $\pp_\mm \in X(\mm)$, showing that $X' \subseteq X$. Conversely, assume that $\pp \in X(\mm)^* \subseteq X$ and let $\mathfrak{m'}$ be another maximal ideal of $R$. Then there are two possibilities. Either $\pp \subseteq \mathfrak{m'}$, and then $\pp \in X(\mathfrak{m'})^*$ by the condition (\ref{gluingcondition}). The other possibility is that $\pp \not\subseteq \mathfrak{m'}$, and then $\pp_\mathfrak{m'} = R_\mathfrak{m'}$. Together we proved that $X = \{\pp \in \Spec(R) \mid V(\pp_\mm) \subseteq X(\mm) ~\forall \mm \in \mSpec(R)\}$. Since $X$ is Thomason, for any $\pp \in X$ there is a finitely generated ideal $I$ of $R$ such that $\pp \in V(I) \subseteq X$. For any $\mm \in \mSpec(R)$ we have $V(I_\mm)  = \{\pp_\mm \mid \pp \in V(I)\}$, and so $V(I_\mm) \subseteq X(\mm)$, as desired.

    $(ii) \implies (i):$ First, the condition $(ii)$ clearly implies that $X$ is Thomason. We are left with proving that the set $\{X(\mm) \mid \mm \in \mSpec(R)\}$ satisfies condition (\ref{gluingcondition}). Let $\pp \in X(\mm)^*$ satisfy $\pp \subseteq \mathfrak{m'}$ and let us show that $\pp \in X(\mathfrak{m'})^*$. By $(ii)$ there is a finitely generated ideal of $R$ such that $\pp \in V(I) \subseteq X$ and such that $V(I_\mathfrak{m'}) \subseteq X(\mathfrak{m'})$. This already testifies that $\pp \in X(\mathfrak{m'})^*$.
\end{proof}
\begin{definition}\label{def:compatible-of-Thomason}
    Let $X(\mm)$ be a Thomason subset of $\Spec(R_\mm)$ for each $\mm \in \mSpec(R)$. The family $\{X(\mm) \mid \mm \in \mSpec(R)\}$ is called \emph{compatible} if it satisfies the equivalent conditions of Lemma~\ref{lem4.9}.

    Let $\mathbb{X}(\mm)  = (X(\mm)_n \mid n \in \Z) $ be a Thomason filtration of $\Spec(R_\mm)$ for each $\mm \in \mSpec(R)$. The family $\{\mathbb{X}(\mm) \mid \mm \in \mSpec(R)\}$ is called \emph{compatible} if $\{X(\mm)_n \mid \mm \in \mSpec(R)\}$ is a compatible family of Thomason subsets for each $n \in \Z$.
\end{definition}
\begin{remark}\label{rem_noeth}
    If $R$ is noetherian then the assumption of $X$ being a Thomason set in condition (i) of Lemma~\ref{lem4.9} is superfluous. Indeed, the proof of the lemma shows that $X = \{\pp \in \Spec(R) \mid V(\pp_\mm) \subseteq X(\mm) ~\forall \mm \in \mSpec(R)\}$ holds even without such assumption, and this is enough to see that $X$ is a specialization closed subset of $\Spec(R)$, which for noetherian rings amounts to $X$ being Thomason. It follows that our compatibility condition generalizes the one used in \cite[Definition 2.3]{tr14}. Indeed, as we just observed, for noetherian rings our compatibility condition for a family of Thomason sets boils down to (\ref{gluingcondition}), which is precisely the condition used by Trlifaj and Şahinkaya.

    On the other hand, the condition of $X$ being Thomason is not superfluous for rings which are not noetherian, in general. Indeed, let $R = k^\omega$ be the countably infinite product of a field $k$ and choose a non-principal maximal ideal $\mm$ of $R$. Set $X(\mm) = \{\mm_\mm\}$ and $X(\mathfrak{m'}) = \emptyset$ for any maximal ideal $\mathfrak{m'} \neq \mm$. Clearly, this family satisfies (\ref{gluingcondition}). However $X = \bigcup_{\mathfrak{n} \in \mSpec(R)} X(\mathfrak{n})^* = \{\mm\}$, which is well-known not to be a Thomason set, as there is no idempotent $e$ of $R$ with $V(e) = \{\mm\}$.
\end{remark}

\begin{proposition}\label{prop_corr_thom}
    There is a bijective correspondence
    $$\left \{ \begin{tabular}{ccc} \text{ Thomason subsets $X$ }  \\ \text{ of $\Spec(R)$ } \end{tabular}\right \}  \xleftrightarrow{1-1}  \left \{ \begin{tabular}{ccc} \text{ Compatible families } \\ \text{ $\{X(\mm) \mid \mm \in \mSpec(R)\}$ of Thomason subsets } \end{tabular}\right \}$$
    induced by the mutually inverse assignments
    $$X \mapsto X(\mm) = X_\mm = \{\pp_\mm \mid \pp \in X\} ~\forall \mm \in \mSpec(R)$$ and
    $$\{X(\mm) \mid \mm \in \mSpec(R)\} \mapsto X = \bigcup_{\mm \in \mSpec(R)}X(\mm)^*.$$
    This bijection naturally extends to a bijection
    $$\left \{ \begin{tabular}{ccc} \text{ Thomason filtrations $\mathbb{X}$ }  \\ \text{ of $\Spec(R)$ } \end{tabular}\right \}  \xleftrightarrow{1-1}  \left \{ \begin{tabular}{ccc} \text{ Compatible families } \\ \text{ $\{\mathbb{X}(\mm) \mid \mm \in \mSpec(R)\}$ of Thomason filtrations }\end{tabular}\right \}.$$
\end{proposition}
\begin{proof}
    Let $X$ be a Thomason subset of $\Spec(R)$ and put $X_\mm = \{\pp_\mm \mid \pp \in X\}$ for each maximal ideal $\mm$. Clearly, $X_\mm^* = X \cap \Spec(R_\mm)^*$ and so the condition (\ref{gluingcondition}) is satisfied. Also, $X = \bigcup_{\mm \in \mSpec(R)}X_\mm^*$ is Thomason, showing that the family $\{X_\mm \mid \mm \in \mSpec(R)\}$ is compatible. It follows from Lemma~\ref{lem4.9} that if a given family $\{X(\mm) \mid \mm \in \mSpec(R)\}$ is compatible, then $X(\mm) = X_\mm = \{\pp_\mm \mid \pp \in X\}$, which establishes that the assignments are mutually inverse.

    For the claim about Thomason filtrations, it is enough to notice that for two Thomason subsets $X,Y$ we have $X \subseteq Y$ if and only if $X_\mm \subseteq Y_\mm$ for each maximal ideal $\mm$, as the sets $\Spec(R_\mm)^*, \mm \in \mSpec(R)$ form a cover of $\Spec(R)$.
\end{proof}
\subsection{Local-global property of compact generation} The gluing condition on Thomason sets allows us to prove the following useful local characterization of when a homotopically smashing t-structure is compactly generated. Indeed, it allows us to construct a family of compact generators out of a family of compact generators for each localization at maximal ideal. Before that, we need a relatively straightforward observation on the injective $R$-modules living in torsion-free classes of hereditary torsion pairs.

Let $\InjR$ denote the subcategory of all injective $R$-modules. We will consider subcategories $\Ecal$ of $\InjR$ which reflect the closure properties of torsion-free classes of hereditary torsion pairs of finite type. Two obvious conditions are that $\Ecal$ needs to be closed under products and direct summands. Furthermore, the fact that the torsion-free class is closed under direct limit will be reflected by the following condition on the subcategory $\Ecal \subseteq \InjR$:
\begin{equation}\label{injectivecondition}\tag{$\dagger\dagger$}
\text{For any direct system $(E_i \mid i \in I)$ in $\Ecal$, the injective envelope of its colimit belongs to $\Ecal$.}
\end{equation}
\begin{lemma}\label{lem_cog_inj}
    There is a bijection
    $$\left \{ \begin{tabular}{ccc} \text{ Hereditary torsion pairs $(\Tcal,\Fcal)$} \\ \text{in $\ModR$} \end{tabular}\right \}  \xleftrightarrow{1-1}  \left \{ \begin{tabular}{ccc} \text{ Subcategories $\Ecal$ of $\InjR$} \\ \text{ closed under products and direct summands } \end{tabular}\right \}$$
    provided by the mutually inverse assignments
    $$\Fcal \mapsto \Ecal = \InjR \cap \Fcal$$
    and
    $$\Ecal \mapsto \Tcal = \{M \in \ModR \mid \Hom_R(M,\Ecal) = 0 \}.$$
    Furthermore, the torsion pair $(\Tcal,\Fcal)$ is of finite type (that is, $\Fcal$ is closed under direct limits) if and only if the corresponding subcategory $\Ecal$ satisfies condition (\ref{injectivecondition}).
\end{lemma}
\begin{proof}
    First, it is routine to check that both the assignments are well-defined. If $(\Tcal,\Fcal)$ is a hereditary torsion pair then $\Fcal$ is closed under injective envelopes, and therefore $\Tcal = \{M \in \ModR \mid \Hom_R(M,\Ecal) = 0 \}$. To establish that the assignments are mutually inverse, the only non-trivial step is to check that if $\Ecal$ is closed under products and direct summands then the corresponding torsion pair $(\Tcal,\Fcal)$ satisfies $\Fcal \cap \InjR \subseteq \Ecal$. To see this, it is enough to prove that $\Fcal = \Sub(\Ecal)$, the subcategory of all submodules of modules from $\Ecal$. Since $\Fcal$ is the smallest torsion-free class containing $\Ecal$, it is further enough to show that $\Sub(\Ecal)$ is a torsion-free class in $\ModR$. Clearly, $\Sub(\Ecal)$ is closed under subobjects, and since $\Ecal$ is closed under products, so is $\Sub(\Ecal)$. Therefore, we only need to see that $\Sub(\Ecal)$ is closed under extensions. Let $0 \rightarrow A_0 \rightarrow B \rightarrow A_1 \rightarrow 0$ be an exact sequence with $A_i \subseteq E_i$ where $E_i$ belongs to $\Ecal$ for both $i=0,1$. The injectivity of $E_0$ allows to extend the inclusion $A_0 \subseteq E_0$ to a map $B \rightarrow E_0$, which easily yields an embedding of $B$ into $E_0 \oplus E_1$. Since $\Ecal$ is closed under products, this shows $B \in \Sub(\Ecal)$, as desired.

    Finally we show that $\Fcal$ is closed under direct limits if and only if $\Ecal$ satisfies (\ref{injectivecondition}). If $\Fcal$ is closed under direct limits then for any direct system $E_i, i \in I$ in $\Ecal$ we have $\varinjlim_{i \in I}E_i \in \Fcal$. Since $\Fcal$ is closed under injective envelopes, the injective envelope of this direct limit belongs to $\Ecal$. Conversely, let $\Ecal$ satisfy (\ref{injectivecondition}). To establish that $\Fcal$ is closed under direct limits, it suffices to show that it is closed under direct limits of well-ordered systems \cite[Lemma 2.14]{gt}. Let $(F_\alpha \mid \alpha < \lambda$) be such a direct system in $\Fcal$ and denote its direct limit by $F = \varinjlim_{\alpha < \lambda}F_\alpha$. By \cite[Lemma 3.5]{hr20}, there is a direct system $E_\alpha, \alpha < \lambda$ such that $E_\alpha$ is the injective envelope of $F_\alpha$ for each $\alpha < \lambda$, and such that the envelope embeddings $F_\alpha \subseteq E_\alpha$ induce a map between the two direct systems. Since $\Ecal$ satisfies (\ref{injectivecondition}), the injective envelope $E$ of the direct limit $M = \varinjlim_{\alpha < \lambda}E_\alpha$ belongs to $\Ecal$. By the properties of the direct systems constructed and by exactness of the direct limit functor, we have a natural limit monomorphism $F \rightarrow M$, and therefore $F$ embeds into $E$. We conclude that $F \in \Fcal$ as desired.
\end{proof}
We are ready to prove a key result of this section showing that the property of a homotopically smashing t-structure being compactly generated is a local-global property with respect to the cover $\Spec(R) = \bigcup_{\mm \in \mSpec(R)}\Spec(R_\mm)^*$.
\begin{proposition}\label{prop_cg_loc}
    Let $(\Ucal,\Vcal)$ be a homotopically smashing t-structure in $\Der(R)$. For each maximal ideal $\mm$, let $(\Ucal_\mm,\Vcal_\mm)$ be the localized t-structure in $\Der(R_\mm)$ of Lemma~\ref{lemma_tstr_loc}. Then:
    \begin{enumerate}
        \item[(i)] The t-structure $(\Ucal,\Vcal)$ is compactly generated in $\Der(R)$ if and only if the t-structure $(\Ucal_\mm,\Vcal_\mm)$ is compactly generated in $\Der(R_\mm)$ for each $\mm \in \mSpec(R)$.
        \item[(ii)] Assume that $(\Ucal,\Vcal)$ is compactly generated and corresponds to a Thomason filtration $\mathbb{X}$. Let $\mathbb{X}(\mm)$ be a Thomason filtration corresponding to the t-structure $(\Ucal_\mm,\Vcal_\mm)$ for each $\mm \in \mSpec(R)$. Then $\{\mathbb{X}(\mm) \mid \mm \in \mSpec(R)\}$ is a compatible family of Thomason filtrations corresponding via Proposition~\ref{prop_corr_thom} to the Thomason filtration $\mathbb{X}$ on $\Spec(R)$.
    \end{enumerate}
\end{proposition}
\begin{proof}
    $(i):$ If $(\Ucal,\Vcal)$ is compactly generated then there is a set $\Scal$ of compact objects such that $\Vcal = \Scal^{\perp_0}$. By Lemma~\ref{lem_adjunction} it is easy to see that $\Vcal_\mm = \{S_\mm \mid S \in \Scal\}^{\perp_0}$. Since $S_\mm$ is a compact object of $\Der(R_\mm)$ for any $S \in \Der(R)^c$, we see that $(\Ucal_\mm,\Vcal_\mm)$ is compactly generated.

    For the converse, let $\mathbb{X}(\mathfrak{m})=(X(\mm)_n \mid n \in \Z)$ be a Thomason filtration corresponding via Theorem~\ref{thmcompgen} to the compactly generated t-structure $(\Ucal_\mm,\Vcal_\mm)$ for each $\mm \in \mSpec(R)$. Put $X_n = \bigcup_{\mm \in \mSpec(R)}(X(\mm)_n)^*$ for each $n \in \Z$. By Lemma~\ref{lem_tstr_techn}(i), $X(\mm)_n = \{\pp \in \Spec(R_\mm) \mid \kappa(\pp)[-n] \in \Ucal_\mm\}$. Since $\Ucal_\mm = \Ucal \cap \Der(R_\mm)$ by Lemma~\ref{lemma_tstr_loc}, we have $(X(\mm)_n)^* = \{\pp\in \Spec(R) \mid \pp \subseteq \mm \text{ and } \kappa(\pp)[-n] \in \Ucal\}$. Then it easily follows that the family $\{X(\mm)_n \mid \mm \in \mSpec(R)\}$ satisfies condition (\ref{gluingcondition}) for each $n \in \Z$. Next, let $\Ecal_n = \{E \in \ModR \mid E[-n] \in \Vcal \text{ and $E$ is injective}\}$ for each $n \in \Z$. Clearly, $\Ecal_n$ is closed under direct summands and products as a subcategory of $\InjR$. We claim that $\Ecal_n$ satisfies condition (\ref{injectivecondition}) for each $n \in \Z$. Indeed, let $(E_i \mid i \in I)$ be a direct system of modules from $\Ecal_n$. Denote its direct limit by $M = \varinjlim_{i \in I}E_i$. Since $E_\alpha \in \Ecal_n$, we have $E_\alpha[-n] \in \Vcal$ for each $n \in \Vcal$. Therefore, $M[-n] = \mathrm{hocolim}_{\alpha < \lambda} E_\alpha[-n]$ belongs to $\Vcal$, as $\Vcal$ is closed under directed homotopy colimits. Let $E$ be the injective envelope of $M$, we need to show that $E[-n] \in \Vcal$. For any $\mm \in \mSpec(R)$, $M_\mm[-n]$ belongs to $\Vcal_\mm$, and so the injective envelope $G_\mm$ of $M_\mm$ in $\ModRm$ also satisfies $G_\mm[-n] \in \Vcal_\mm$, here we use Lemma~\ref{lem_tstr_techn}(ii). Since $G_\mm$ is also injective as an $R$-module, we have $G_\mm \in \Ecal_n$. Finally, observe that $E$ is a direct summand of $\prod_{\mm \in \mSpec(R)}G_\mm$, and so $E \in \Ecal_n$.

    Let $(\Tcal_n, \Fcal_n)$ be the hereditary torsion pair of finite type in $\ModR$ corresponding to $\Ecal_n$ via Lemma~\ref{lem_cog_inj}, and let $X'_n$ be the Thomason set in $\Spec(R)$ corresponding to $(\Tcal_n,\Fcal_n)$ via Theorem~\ref{thm_torsion_pair}. Note that since $\Tcal_n = \{M \in \ModR \mid \Hom_R(M,\Ecal_n) = 0 \}$ and $\Supp(\kappa(\pp)) = \{\pp\}$, we have $X'_n = \{\pp \in \Spec(R) \mid \Hom_R(\kappa(\pp),\Ecal_n) = 0\}$. We claim that $X'_n = X_n$. The inclusion $X_n \subseteq X'_n$ is clear as $\Ecal_n[-n] \subseteq \Vcal$. On the other hand, if $\pp \in X'_n$ then consider any maximal ideal $\mm$ which contains $\pp$. Since $\Hom_R(\kappa(\pp),\Ecal_n) = 0$, we have in particular that $\Hom_R(\kappa(\pp),\Ecal_n(\mm)) = 0$ where $\Ecal_n(\mm)$ consists of all injective $R_\mm$-modules $E$ such that $E[-n] \in \Vcal_\mm$. But since the t-structure $(\Ucal_\mm,\Vcal_\mm)$ is compactly generated, this implies $\pp_\mm \in X_n^\mm$, and so $\pp \in (X_n^\mm)^* \subseteq X_n$ by Lemma~\ref{lem_tstr_techn}(iv).

    We have proved that $\{\mathbb{X}(\mm) \mid \mm\in \mSpec(R)\}$ is a compatible family of Thomason filtrations corresponding to a Thomason filtration $\mathbb{X} = \{X_n \mid n \in \Z\}$ via Proposition~\ref{prop_corr_thom}. Put $\Scal = \{K(I)[-n] \mid V(I) \subseteq X_n, n \in \Z\}$. It follows from Theorem~\ref{thmcompgen} that $\Scal_\mm = \{K(I)[-n] \otimes_R R_\mm \mid V(I) \subseteq X_n, n \in \Z\} = \{K(I_\mm)[-n] \mid V(I_\mm) \subseteq X(\mm)_n, n \in \Z\}$ is a set of compact generators for $(\Ucal_\mm,\Vcal_\mm)$ in $\Der(R_\mm)$, and thus satisfies $(\Scal_\mm)^{\perp_0} = \Vcal_\mm$. By Lemma~\ref{lem_adjunction}, we have $(\Scal^{\perp_0})_\mm = (\Scal_\mm)^{\perp_0} = \Vcal_\mm$. But since both $\Scal^{\perp_0}$ and $\Vcal$ are definable subcategories of $\Der(R)$, Lemma~\ref{lem_def_loc} shows that $\Scal^{\perp_0} = \Vcal$, and so $(\Ucal,\Vcal)$ is compactly generated.

    $(ii):$ This follows easily either from the proof of $(i)$ or directly from Lemma~\ref{lem_tstr_techn}(i).
\end{proof}
\subsection{Gluing of compactly generated t-structures}
\begin{definition}
    Let $(\Ucal(\mm),\Vcal(\mm))$ be a compactly generated t-structure in $\Der(R_\mm)$ for each $\mm \in \mSpec(R)$ corresponding to a Thomason filtration $\mathbb{X}(\mm)$ in $\Spec(R_\mm)$. The family $\{(\Ucal(\mm),\Vcal(\mm)) \mid \mm \in \mSpec(R)\}$ is said to be \emph{compatible} if the family $\{\mathbb{X}(\mm) \mid \mm \in \mSpec(R)\}$ of Thomason filtrations is compatible.
\end{definition}

We say that a Thomason filtration $\mathbb{X} = (X_n \mid n \in \Z)$ is \emph{non-degenerate} if $\bigcap_{n \in \Z}X_n = \emptyset$ and $\bigcup_{n \in \Z}X_n = \Spec(R)$.

\begin{theorem}\label{thm-compact-t-structure}Let $R$ be a commutative ring. There is a bijective correspondence between
\begin{enumerate}
\item[(i)] compactly generated t-structures $(\Ucal,\Vcal)$ in $\Der(R)$, and

\item[(ii)] compatible families $\{(\Ucal(\mm),\Vcal(\mm)) \mid \mm \in \mSpec(R)\}$ of compactly generated t-structures,
\end{enumerate}

 which restricts to a bijective correspondence between
\begin{enumerate}
\item[(i')] non-degenerate compactly generated t-structures $(\Ucal,\Vcal)$ in $\Der(R)$, and

\item[(ii')] compatible families $\{(\Ucal(\mm),\Vcal(\mm)) \mid \mm \in \mSpec(R)\}$ of non-degenerate compactly generated t-structures.
\end{enumerate}
\end{theorem}
\begin{proof} The bijection between $(i)$ and $(ii)$ follows from Theorem \ref{thmcompgen}, Proposition \ref{prop_corr_thom}, and Proposition \ref{prop_cg_loc}$(ii)$. Recall that $(\Ucal,\Vcal)$ is a non-degenerate t-structure if it satisfies $\bigcap_{n \in \Z}\Ucal[n] = 0 = \bigcap_{n \in \Z}\Vcal[n]$. To prove the equivalence of $(i')$ and $(ii')$, by the bijection between $(i)$ and $(ii)$, it suffices to show that a compactly generated t-structure $(\Ucal,\Vcal)$ is non-degenerate if and only if the corresponding Thomason filtration $\mathbb{X}$ corresponding to $(\Ucal,\Vcal)$ is non-degenerate.

    Let $\Ucal_\# = \{X \in \Der(R) \mid \Supp H^n(X) \subseteq X_n, n
\in \Z\}$ be the aisle of Lemma~\ref{lem_bounded_dev}, recall that $\Ucal \subseteq \Ucal_\#$. Now if $\bigcap_{n \in \Z}X_n = \emptyset$ then $\bigcap_{n \in \Z} \Ucal_\#[n] = \{X \in \Der(R) \mid \Supp H^*(X) = \emptyset\} = 0$, which imples $\bigcap_{n \in \Z} \Ucal[n] = 0$. On the other hand, assume that $\bigcap_{n \in \Z}X_n$ contains a prime ideal $\pp$. Then $\Ucal[n]$ contains $\kappa(\pp)[0]$ for each $n \in \Z$ by Lemma~\ref{lem_tstr_techn}(i), and so $\bigcap_{n \in \Z}\Ucal[n] \neq 0$.

Now we consider the second non-degeneracy condition. Put $\Scal = \Ucal \cap \Der(R)^c$. Since the t-structure is compactly generated we have $\Scal^{\perp_0} = \Vcal$. Observe that $\bigcap_{n \in \Z}\Vcal[n] = 0$ is equivalent to $\Scal^{\perp_\Z} = (\bigcup_{n \in \Z}\Scal[n])^{\perp_0} = 0$. We let $(\Lcal,\Ccal)$ be the t-structure generated by the set $\bigcup_{n \in \Z}\Scal[n]$, so that $\Ccal = \Scal^{\perp_\Z}$. The t-structure $(\Lcal,\Ccal)$ is compactly generated and so it corresponds to a Thomason filtration $\mathbb{Y} = (Y_n \mid n \in \Z)$ via Theorem~\ref{thmcompgen}. Since the t-structure $(\Lcal,\Ccal)$ is stable (recall that this means that $\Lcal$ is closed under the cosuspension functor $[-1]$), the Thomason filtration $\mathbb{Y}$ is constant in the sense that there is a single Thomason set $Y$ such that $Y_n = Y$ for all $n \in \Z$, cf. \cite[Theorem 5.3]{hr20}. It also follows from Theorem~\ref{thmcompgen} that $Y = \bigcup_{n \in \Z}X_n$, as $\Lcal$ is generated by all shifts of the Koszul complexes of the form $K(I)$ with $V(I) \subseteq X_n$ for any $n \in \Z$. Finally, we have that $\Ccal = 0$ if and only if $\Lcal = \Der(R)$, which happens if and only if the Thomason set $Y$ is equal to $\Spec(R)$.
\end{proof}

\section{Stalk-locality of Telescope Conjecture}
\begin{definition}\label{def_tc}
    A triangulated category $\Tcal$ satisfying the assumptions of \S\ref{S:prelim} is said to
    \begin{itemize}
        \item satisfy the \emph{Telescope Conjecture (TC)} if any \emph{stable} (homotopically) smashing t-structure in $\Tcal$ is compactly generated.
        \item satisfy the \emph{Semistable Telescope Conjecture (STC)} if any homotopically smashing t-structure in $\Tcal$ is compactly generated.
    \end{itemize}
\end{definition}
\begin{remark}
    The formulation (TC) of Definition~\ref{def_tc} is equivalent to a more customary form of the Telescope Conjecture which asks for any kernel of a smashing localization of the triangulated category $\Der(R)$ to be generated by compact objects. More precisely, a stable t-structure $(\Lcal,\Ccal)$ is homotopically smashing if and only if $\Ucal$ is a \emph{smashing subcategory}, that is, if $\Ucal$ is a localizing subcategory of $\Der(R)$ such that $\Ucal^{\perp_0} = \Vcal$ is closed under coproducts, see e.g. \cite[Appendix]{hn20} for details and further references.

    Clearly, the validity of (STC) in $\Der(R)$ implies (TC).
\end{remark}

The fact that (TC) holds in the derived category of a commutative noetherian ring was established by Neeman \cite{ne1992}. The following recent development shows that also (STC) is valid in the same setting.
\begin{theorem}\label{thmHN}\emph{(\cite[Theorem 1.1]{hn20})}
    Let $R$ be a commutative noetherian ring. Then $\Der(R)$ satisfies (STC).
\end{theorem}
We are ready to formulate the local-global property of Proposition~\ref{prop_cg_loc} in terms of telescope properties of $\Der(R)$.
\begin{lemma}\label{lem_tc_loc}
    Let $R$ be a commutative ring and $\pp \in \Spec(R)$. Then:
    \begin{enumerate}
        \item[(i)] If (TC) holds in $\Der(R)$ then (TC) holds in $\Der(R_\pp)$.
        \item[(ii)] If (STC) holds in $\Der(R)$ then (STC) holds in $\Der(R_\pp)$.
    \end{enumerate}
\end{lemma}
\begin{proof}
    Assume that (STC) holds in $\Der(R)$ and let $\pp \in \mSpec(R)$. Let $(\Ucal,\Vcal)$ be a homotopically smashing t-structure in $\Der(R_\pp)$. Consider $\Vcal$ as a subcategory of $\Der(R)$. Since $\Vcal$ is a cosuspended definable subcategory of $\Der(R_\pp)$ (see \cite[Remark 8.9]{ss20}) and $\Der(R_\pp)$ is a cosuspended definable subcategory of $\Der(R)$, we infer that $\Vcal$ is a cosuspended definable subcategory of $\Der(R)$, and so there is a homotopically smashing t-structure $(\Ucal',\Vcal)$ in $\Der(R)$, see \cite[Proposition 4.5]{li17}. Since $\Der(R)$ satisfies (STC), there is a set of compact objects $\Scal$ of $\Der(R)$ such that $\Vcal = \Scal^{\perp_0}$ in $\Der(R)$. By Lemma~\ref{lem_adjunction}, we have the equality $\Vcal = \Scal_\pp^{\perp_0}$ in $\Der(R_\pp)$, and so $(\Ucal,\Vcal)$ is compactly generated in $\Der(R_\pp)$. The version of this implication for (TC) follows easily as the closure of $\Vcal$ under suspensions is also checked equivalently in $\Der(R_\pp)$ and $\Der(R)$.
\end{proof}
\begin{theorem}\label{thm_tc_loc}
    Let $R$ be a commutative ring.
        \begin{enumerate}
            \item[(i)] (TC) holds in $\Der(R)$ if and only if (TC) holds in $\Der(R_\mm)$ for any maximal ideal $\mm$ of $R$.
            \item[(ii)] (STC) holds in $\Der(R)$ if and only if (STC) holds in $\Der(R_\mm)$ for any maximal ideal $\mm$ of $R$.
        \end{enumerate}
\end{theorem}
\begin{proof}
    By Lemma~\ref{lem_tc_loc}, we only need to prove the backward implication of both statements. If a t-structure $(\Ucal,\Vcal)$ is homotopically smashing in $\Der(R)$ then clearly so is $(\Ucal_\mm,\Vcal_\mm)$ in $\Der(R_\mm)$. Therefore, if (STC) holds for each localization $R_\mm$ at maximal ideals then (STC) holds in $\Der(R)$ by Proposition~\ref{prop_cg_loc}(i). Also, if $(\Ucal,\Vcal)$ is in addition stable then so is clearly $(\Ucal_\mm,\Vcal_\mm)$. As a consequence, we also get that if (TC) holds for each localization $R_\mm$ then it holds in $\Der(R)$.
\end{proof}
\begin{corollary}\label{cor_loc_noeth}
    Let $R$ be a commutative ring such that $R_\mm$ is noetherian for any $\mm \in \mSpec(R)$. Then (STC) holds in $\Der(R)$. In particular, (TC) holds in $\Der(R)$.
\end{corollary}
\begin{proof}
    Combine Proposition~\ref{prop_cg_loc} and the main result of \cite{hn20}.
\end{proof}
\begin{remark}\label{rem_nonnoeth}
    There are many examples of non-noetherian commutative rings which are locally noetherian in the sense of Corollary~\ref{cor_loc_noeth}, see e.g. \cite{ho71}, \cite{ab}. A well-known class of examples of such rings consists of the (commutative and not semi-simple) von Neumann regular rings --- indeed, these can be characterized as rings $R$ such that $R_\mm$ is a field for any $\mm \in \Spec(R)$. In this way, our local-global criterion recovers some recent results for von Neumann regular rings \cite[\S 7]{bs17}, \cite{st14}, \cite[Corollary 3.12]{bh20}.
\end{remark}
\subsection{Stalk-locality of $\otimes$-Telescope Conjecture for schemes}\label{ss_stalk}
Here we follow the setting of \cite[Examples 1.2(2)]{bf11}, for basic terminology about schemes we adhere to \cite{vakil}. Let $X$ be a quasi-compact and quasi-separated scheme $X$ and let $\Der(X)$ denote the derived category of complexes of $\Ocal_X$-modules with quasi-coherent cohomology. Then $\Der(X)$ together with the usual derived tensor product $\otimes^\mathbf{L}_X$ is a compactly generated tensor triangulated category. If $X$ is in addition separated then $\Der(X)$ is naturally equivalent to $\Der(\mathrm{Qcoh(X)})$, the usual derived category of the category of quasi-coherent sheaves over $X$. For non-affine schemes, one can only expect good local behavior from localizations which respect the tensor structure. The following formulation of tensor-friendly Telescope Conjecture is a special case of the general formulation for tensor triangulated categories of Balmer and Favi \cite[Definition 4.2]{bf11}, building on previous work of Hovey, Palmieri and Strickland \cite{hps}.
\begin{definition}
Let $X$ be a quasi-compact and quasi-separated scheme. We say that $\Der(X)$ satisfies the \emph{$\otimes$-Telescope Conjecture ($\otimes$TC)} if any stable (homotopically) smashing t-structure $(\Lcal,\Ccal)$ such that $\Lcal$ is a $\otimes$-ideal, meaning that $\Lcal$ satisfies in addition the condition
\begin{equation}\label{E:tensor}L \otimes_X^\mathbf{L} M \in \Lcal ~\text{ for all } L \in \Lcal, M \in \Der(X),\end{equation}
is compactly generated. It is well-known that in case $X$ is an affine scheme the tensor condition (\ref{E:tensor}) is vacuous, and so for affine schemes, ($\otimes$TC) is equivalent to (TC) by taking the ring of global sections.
\end{definition}
Balmer and Favi \cite[Corollary 6.8]{bf11} showed that ($\otimes$TC) is a local-global property with respect to any cover $X = \bigcup_{i \in I}U_i$ by open and quasi-compact sets. In particular, one can check ($\otimes$TC) locally on any cover of $X$ by open affine sets. We remark that Balmer and Favi worked in a much broader generality of compactly generated tensor triangulated categories in terms of Balmer spectra. This was further generalized by Stevenson \cite{st13} to the relative setting of a suitable action  triangulated categories. In algebrogeometric context, Antieau \cite{an14} established the local-global criterion for \'{e}tale covers of (derived) schemes. Note these results do not include the case of the cover $\Spec(R) = \bigcup_{\mm \in \mSpec(R)}\Spec(R_\mm)^*$ we consider in Theorem~\ref{thm_tc_loc}$(i)$ --- indeed, this cover is not even fpqc whenever $\mSpec(R)$ is an infinite set.

Given a scheme $X$ and $x \in X$, we denote the \emph{stalk} of $X$ at $x$ by $\Ocal_{X,x}$, recall that $\Ocal_{X,x}$ is always a local commutative ring. Combining Theorem~\ref{thm_tc_loc}$(i)$ with the Balmer and Favi result, we obtain that ($\otimes$TC) for $\Der(X)$ is a \emph{stalk-local} property in the following sense.
\begin{theorem}\label{tc_stalk_local}
    Let $X$ be a quasi-compact and quasi-separated scheme. Then the following statements are equivalent:
    \begin{enumerate}
        \item[(i)] ($\otimes$TC) holds in $\Der(X)$,
        \item[(ii)] ($\otimes$TC), or equivalently (TC), holds in $\Der(\Ocal_{X,x})$ for all $x \in X$,
        \item[(iii)] ($\otimes$TC), or equivalently (TC), holds in $\Der(\Ocal_{X,x})$ for all closed points $x \in X$.
    \end{enumerate}
\end{theorem}
\begin{proof}
    Let $X = \bigcup_{i \in I} U_i$ be a cover of $X$ by open affine sets, let $\lambda_i: U_i \cong \Spec(R_i)$ be homeomorphisms where $R_i$ is a appropriate commutative ring for each $i \in I$. By \cite[Corollary 6.8]{bf11}, ($\otimes$TC) holds in $\Der(X)$ if and only if ($\otimes$TC) holds in $\Der(U_i)$ for each $i \in I$. Then ($\otimes$TC) holds in $\Der(X)$ if and only if (TC) holds in $\Der(R_i)$ for each $i \in I$. By Theorem~\ref{thm_tc_loc}$(i)$ we know that (TC) holds in $\Der(R_i)$ if and only if it holds in $\Der((R_i)_{\pp})$ for any $\pp \in \Spec(R_i)$. But $\Der((R_i)_{\pp}) \cong \Der(\Ocal_{U_i,\lambda_i(\pp)}) \cong \Der(\Ocal_{X,\lambda_i(\pp)})$. This establishes the equivalence $(i) \iff (ii)$.

    It remains to prove the implication $(iii) \implies (ii)$. Let $x \in X$ be any point. Since $X$ is quasi-compact, there is a closed point $c$ contained in the closure of $x$ in $X$. By the assumption, (TC) holds in $\Der(\Ocal_{X,c})$. Since the ring $\Ocal_{X,x}$ is isomorphic to a localization of $\Ocal_{X,c}$ at some prime ideal, Lemma~\ref{lem_tc_loc} implies that (TC) holds in $\Der(\Ocal_{X,x})$.
\end{proof}
As a consequence, ($\otimes$TC) holds if $X$ is stalk-noetherian in the following sense. As mentioned already in Remark~\ref{rem_nonnoeth}, the property of a scheme being noetherian is not a stalk-local property, even for affine schemes \cite{ho71},\cite{ab}.
\begin{corollary}\label{cor_loc_noeth_scheme}
    Let $X$ be a quasi-compact and quasi-separated scheme such that the stalk $\Ocal_{X,x}$ is noetherian for any point $x \in X$ (equivalently, for any closed point $x \in X$). Then $\Der(X)$ satisfies ($\otimes$TC).
\end{corollary}
\begin{proof}
    This is a direct consequence of Theorem~\ref{tc_stalk_local} together with validity of (TC) in derived categories of commutative noetherian rings \cite{ne1992}.
\end{proof}
\section{Cosilting and cotilting objects}\label{S:cosilting}
The goal of this section is to refine the (co)localization results for t-structures that are induced by cosilting objects. \begin{definition}We say that an object $T \in \Der(R)$ is \emph{silting} if the pair
    $(T^{\perp_{>0}}, T^{\perp_{\leq0}})$ is a t-structure, which we call the
    \emph{silting t-structure} induced by $T$. Two silting objects
    $T, T'\in \Der(R)$ are\emph{ equivalent} if they induce the same t-structure.

    An object $C\in \Der(R)$ is \emph{cosilting} if the pair $(^{\perp_{\leq0}}C, {^{\perp_{>0}}C})$ forms a t-structure, which
    we call the \emph{cosilting t-structure} induced by $C$. Two cosilting objects $C,C'$ are \emph{equivalent} if they induce the same t-structure. By \cite[Lemma 4.5]{ps18}, $C$ and $C'$ are equivalent if and only if $\Prod(C) = \Prod(C')$, where $\Prod(C)$ is the subcategory of all direct summands of set-indexed direct products of copies of $C$.

    A silting object is called a \emph{bounded silting complex }if it is quasi-isomorphic to a bounded complex of projective $R$-modules, and a cosilting object is called a \emph{bounded cosilting complex} if it is quasi-isomorphic to a bounded complex of injective $R$-modules.
    \end{definition}
    It is an easy consequence of the definition that both silting and cosilting t-structures are always non-degenerate. If $C$ is a pure-injective cosilting object then the induced t-structure $(\Ucal,\Vcal)$ is homotopically smashing. Indeed, the pure-injectivity of $C$ ensures that the coaisle $\Vcal = {}^{\perp_{>0}}C$ is closed under both pure monomorphisms and pure epimorphisms.
    %\begin{definition}If $T$ is a silting complex in $\Der(R)$, then we say that $T^{\perp_{>0}}$ is a \emph{silting class. }If $C$ is a cosilting complex in $\Der(R)$, then we say that $^{\perp_{>0}}C$ is a\emph{ cosilting class.}
    %\end{definition}

    %It is well known that a silting (respectively, cosilting) class determines a unique silting (respectively, cosilting) object up to equivalence.

    %\begin{theorem}\cite[Theorem 2.12]{lihrbek} Let $(\Ucal, \Vcal)$ be a $t$-structure in $D(\mathrm{Mod}\text{-}R)$.

    %(1) $\mathcal{(U, V)}$ is induced by a bounded silting complex if and only if there is a intermediate suspended TTF triple $\mathcal{(U, V,W)}$.

    %(2) $\mathcal{(U, V)}$ is induced by a bounded cosilting complex if and only if there is %a cointermediate cosuspended TTF triple $\mathcal{(U, V,W)}$. In particular, $\mathcal{(U, V)}$ is then homotopically smashing.
    %\end{theorem}

    \begin{definition} We say that a silting object $T$ in $\Der(R)$ is of \emph{finite type} if there is a set of compact objects $\Scal$ such that $T^{\perp_{>0}} = \Scal^{\perp_0}$. Similarly, we call a cosilting object $C$ in $\Der(R)$ of\emph{ cofinite type} if there is a set of compact objects $\Scal$ such that ${}^{\perp_{>0}}C = \Scal^{\perp_0}$. Note that the a cosilting object is cofinite type precisely when the induced cosilting t-structure $(\Ucal,\Vcal)$ is compactly generated.
    \end{definition}

    \begin{theorem}\emph{(\cite{ma18})}\label{mv_bounded}
        Any bounded silting object is of finite type. Any bounded cosilting object is pure-injective.
    \end{theorem}

    \begin{theorem}\emph{(\cite[A.8, Corollary 2.14]{hn20})}\label{noeth_coftype}
        If $\Der(R)$ satisfies (STC) then any pure-injective cosilting object is of cofinite type. In particular, this is the case if $R$ is a commutative noetherian ring.
    \end{theorem}

    %\begin{corollary} Let $A$ be a ring. Then:

    %(i) A compactly generated TTF triple is silting if and only if it is suspended and non-degenerate.

    %(ii) A compactly generated TTF triple is cosilting if and only if it is cosuspended and non-degenerate.
    %\end{corollary}

\begin{lemma}\label{lem4.1}
    Let $C \in \Der(R)$. Then $C$ is a cosilting object if and only if all the following conditions hold:

    (i) $C$ cogenerates $\Der(R)$, that is, ${}^{\perp_\Z}C = 0$,

    (ii) $C \in {}^{\perp_{>0}}C$,

    (iii) ${}^{\perp_{>0}}C$ is closed under products,

    (iv) ${}^{\perp_{>0}}C$ is a coaisle of a t-structure.

    Furthermore, the condition (iv) follows from the other three conditions provided that $C$ is pure-injective.
\end{lemma}
\begin{proof}
Using condition (iv) the proof is dual to that of \cite[Proposition 4.13]{ps18}.

If $C$ is pure-injective, ${}^{\perp_{>0}}C$ is closed under both pure monomorphisms and pure epimorphisms. Since ${}^{\perp_{>0}}C$ is clearly cosuspended and by (iii) it is closed under products, we infer by \cite[Lemma 4.8]{li17} that $\Vcal = {}^{\perp_{>0}}C$ is a coaisle of a t-structure.
\end{proof}
We are ready to investigate colocalization properties of cosilting objects in $\Der(R)$.
\begin{lemma}\label{lem_cosilt_coloc}
    Let $C$ be a cosilting object in $\Der(R)$ and $\pp$ a prime ideal of $R$. Then $C^{\pp}$ is a cosilting object in $\Der(R_\pp)$. Furthermore, if $(\Ucal,\Vcal)$ is the t-structure induced by $C$ then $C^\pp$ induces the cosilting t-structure $(\Ucal_\pp,\Vcal^\pp)$ in $\Der(R_\pp)$.
\end{lemma}
\begin{proof}
    By Lemma~\ref{lem_adjunction}, for any object $X \in \Der(R_\pp)$ we have $\Hom_{\Der(R_\pp)}(X,C^\pp) \cong \Hom_{\Der(R)}(X,C)$. From this we easily derive the property (i) and (iii) of Lemma~\ref{lem4.1} applied to the object $C^\pp$ of the category $\Der(R_\pp)$. Also, we can infer from \cite[Proposition 2.2]{hr20} that $C^{\pp} \in {}^{\perp_{>0}}C$, and so the adjunction formula of Lemma~\ref{lem_adjunction} also yields condition (ii). We are left to show the condition (iv), that is,  the subcategory ${}^{\perp_{>0}}C^\pp$ is a coaisle in $\Der(R_\pp)$.

    We denote $\Vcal = {}^{\perp_{>0}}C$ and $\Ucal = {}^{\perp_0}\Vcal$ so that $(\Ucal,\Vcal)$ is the t-structure induced by the cosilting object $C$ in $\Der(R)$. Recall from Lemma~\ref{lemma_tstr_loc} that $\Vcal^\pp = \Vcal \cap \Der(R_\pp)$ and note that $\Vcal^\pp$ is equal as a subcategory of $\Der(R_\pp)$ to ${}^{\perp_{>0}}C^\pp$. But $\Vcal_\pp$ is a coaisle of a t-structure by Lemma~\ref{lemma_tstr_loc}, as desired.
\end{proof}
\begin{lemma}\label{lem_cos_prod}
    Let $C$ be a pure-injective cosilting object in $\Der(R)$. Then $D = \prod_{\mm \in \mSpec(R)}C^\mm$ is a cosilting object in $\Der(R)$ which is equivalent to $C$.
\end{lemma}
\begin{proof}
    We will again use Lemma~\ref{lem4.1} to show that $D$ is a cosilting object. Set $\Vcal = {}^{\perp_{>0}}C$. The main observation we make is that for any $X \in \Der(R)$ we have $X \in {}^{\perp_{>0}}D$ if and only if $X_\mm \in \Vcal$ for each maximal ideal $\mm$. This already shows that $D$ is a cogenerator in $\Der(R)$. Because $\Vcal$ is definable, Lemma~\ref{lem_def_loc} yields that $X \in \Vcal$ whenever $X_\mm \in \Vcal$ for all $\mm \in \mSpec(R)$, and so $\Vcal = {}^{\perp_{>0}}D$. Because $C$ is pure-injective, so is $C^\mm$ for each maximal ideal $\mm$, and therefore $D$ is pure-injective. Finally, $C^\mm \in \Vcal$ for each maximal ideal $\mm$ by \cite[Proposition 2.2]{hr20}, and thus $D \in \Vcal$. Now we can apply Lemma~\ref{lem4.1} to infer that $D$ is a cosilting object, and since $\Vcal = {}^{\perp_{>0}}D$, we have that $D$ is equivalent to $C$.
\end{proof}
\begin{definition}
A cosilting object $C$ is \emph{cotilting} if $\mathrm{Prod}(C) \subseteq {}^{\perp_{<0}}C$.
\end{definition}
\begin{remark}
The importance of the last definition comes from derived equivalences, here we follow \cite{ps18}. Let $(\Ucal,\Vcal)$ be the t-structure induced by $C$ and let $\Hcal = \Vcal \cap \Ucal[-1]$ be the \emph{heart} of the t-structure. Assume that $C$ is a bounded cosilting complex. Then $\Hcal$ is a Grothendieck category \cite{li17}, and there is a \emph{realization functor} $\mathrm{real}_C: \Der^b(\Hcal) \rightarrow \Der^b(R)$ between the bounded derived categories \cite{ps18}. Then the cosilting object $C$ is cotilting if and only if $\mathrm{real}_C$ is an equivalence \cite[Corollary 5.2]{ps18}.
\end{remark}
\begin{lemma}\label{lem4.4}
    If $C$ is a cotilting object in $\Der(R)$ then $C^\mm$ is a cotilting object for any $\mm \in \mSpec(R)$. Furthermore, if $C$ is pure-injective then $\prod_{\mm \in \mSpec(R)}C^\mm$ is a cotilting object in $\Der(R)$ equivalent to $C$.
\end{lemma}
\begin{proof}
    By Lemma~\ref{lem_cos_prod}, the cosilting object $C$ is equivalent to $D = \prod_{\mm \in \mSpec(R)}C^\mm$, which means that $\Prod(C) = \Prod(D)$. Therefore, $\Prod(C^\mm) \subseteq {}^{\perp_{<0}}C$, which by the usual adjunction argument translates to $\Prod(C^\mm) \subseteq {}^{\perp_{<0}}C^\mm$ in $\Der(D_\mm)$.

    If $C$ is pure-injective then $\prod_{\mm \in \mSpec(R)}C^\mm$ is cosilting in $\Der(R)$ by Lemma~\ref{lem_cos_prod}. Then it remains to note that the equivalence between cosilting objects clearly preserves the cotilting property. Indeed, if $C'$ and $C''$ are two equivalent cosilting objects then $\Prod(C) = \Prod(C')$, and therefore $\Prod(C) \subseteq {}^{\perp_{<0}}C$ if and only if $\Prod(C') \subseteq {}^{\perp_{<0}}C'$.
\end{proof}
However, it is not true in general that if $C^\mm$ is a cotilting object for any $\mm \in \mSpec(R)$ then the cosilting object $C$ has to be cotilting as demonstrated in the following example.
\begin{example}
    Let $k$ be a field and $R = k^\omega$ be the countably infinite product of $k$. Recall that $R$ is a von Neumann regular ring, and therefore in particular every simple $R$-module is injective. Let $\mm$ be any maximal ideal of $R$ such that $\Hom_R(R/\mm,R) = 0$. Such maximal ideals are plentiful --- recall that maximal ideals of $R$ are in a natural bijection with ultrafilters on $\omega$ and the desired property of $\mm$ is satisfied if and only if the corresponding ultrafilter is not principal. We also set $\Mcal = \mSpec(R) \setminus \{\mm\}$.

    Put $C = R/\mm [0] \oplus \prod_{\mathfrak{n} \in \Mcal}R/\mathfrak{n}[-1]$ and we claim that $C$ is a cosilting object in $\Der(R)$. Since $C$ is a product of two shifted stalk complexes of injective $R$-modules, we infer that $C$ is pure-injective object of $\Der(R)$. By the injectivity, the orthogonal $\Vcal = {}^{\perp_{>0}}C$ is determined on cohomology and can be easily computed: $\Vcal = \{X \in \Der^{\geq 0} \mid H^0(X) \in \Fcal\}$, where $\Fcal = \{M \in \ModR \mid \Hom_R(M,\prod_{\mathfrak{n} \in \Mcal}R/\mathfrak{n})=0\}$. It is easy to see that $\Fcal = \mathrm{Add}(R/\mm) = \mathrm{Prod}(R/\mm)$. It follows that $C \in \Vcal$ and that $\Vcal$ is product-closed. By Lemma~\ref{lem4.1}, $C$ is a cosilting object.

    Next we show that $C$ is not a cotilting object. From $\Hom_R(R/\mm,R) = 0$ it follows that the intersection of all maximal ideals belonging to $\Mcal$ is zero. Therefore, $\prod_{\mathfrak{n} \in \Mcal}R/\mathfrak{n}$ contains a copy of $R$ as a submodule. Since $R/\mm$ is injective, this yields $\Hom_R(\prod_{\mathfrak{n} \in \Mcal}R/\mathfrak{n},R/\mm) \neq 0$, and therefore there is a non-zero map $C \rightarrow C[-1]$ in $\Der(R)$, witnessing that $C$ is not cotilting.

    Finally, let $\mathfrak{n}$ be a maximal ideal. Then the colocalization $C^\mathfrak{n}$ is equal either to $R/\mm[0]$ in case $\mathfrak{n} = \mm$ or to $R/\mathfrak{n}[-1]$ in case $\mathfrak{n} \neq \mm$. In either case, $C^\mathfrak{n}$ is a shift of the injective cogenerator of the category of vector spaces over the field $R/\mathfrak{n} = R_\mathfrak{n}$, and so $C^\mathfrak{n}$ is a cotilting object in $\Der(R_\mathfrak{n})$.
\end{example}
\subsection{Cofinite type and compatible families of Thomason filtrations}
    We start by generalizing \cite[Theorem 3.8]{lihrbek} and characterize the Thomason filtrations which are induced by a cosilting object. %We say that a Thomason filtration $\mathbb{X} = (X_n \mid n \in \Z)$ is \emph{non-degenerate} if $\bigcap_{n \in \Z}X_n = \emptyset$ and $\bigcup_{n \in \Z}X_n = \Spec(R)$.
\begin{proposition}\label{prop_nondegen_cosilting}
    Let $R$ be a commutative ring. There is a bijective correspondence between the following families:

    (i) equivalence classes of cosilting objects of cofinite type in $\Der(R)$, and

    (ii) non-degenerate Thomason filtrations $\mathbb{X} = (X_n \mid n \in \Z)$ on $\Spec(R)$.

    The correspondence assigns to a cosilting object $C$ of cofinite type the Thomason filtration associated to the compactly generated t-structure induced by $C$ via Theorem~\ref{thmcompgen}.
\end{proposition}
\begin{proof}
    Let $(\Ucal,\Vcal)$ be a compactly generated t-structure in $\Der(R)$ corresponding to a Thomason filtration $\mathbb{X}$. By \cite[Theorem 4.6]{la19}, the t-structure $(\Ucal,\Vcal)$ is induced by a cosilting object if and only if it is non-degenerate. So the result holds by the proof Theorem \ref{thm-compact-t-structure}.
\end{proof}
\begin{lemma}\label{lem_corr_thom_nd}
    The bijection of Proposition~\ref{prop_corr_thom} restricts to a bijection
    $$\left \{ \begin{tabular}{ccc} \text{ Non-degenerate Thomason }  \\ \text{ filtrations $\mathbb{X}$ of $\Spec(R)$ } \end{tabular}\right \}  \xleftrightarrow{1-1}  \left \{ \begin{tabular}{ccc} \text{ Compatible families $\{\mathbb{X}(\mm) \mid \mm \in \mSpec(R)\}$ } \\ \text{  of non-degenerate Thomason filtrations }\end{tabular}\right \}.$$
\end{lemma}
\begin{proof}
    The condition (\ref{gluingcondition}) (see \S\ref{ss_comp_thom}) ensures that $X_n = \bigcup_{\mm \in \mSpec(R)}X(\mm)_n^*$ and $X(\mm)_n^* = X_n \cap \Spec(R_\mm)^*$ for each $n \in \Z$. This already implies that $\bigcup_{n \in \Z}X_n = \Spec(R)$ if and only if $\bigcup_{n \in \Z}X(\mm)_n = \Spec(R_\mm)$ for each $\mm \in \mSpec(R)$ as well as that $\bigcap_{n \in \Z} X_n = \emptyset$ if and only if $\bigcap_{n \in \Z}X(\mm)_n = \emptyset$ for each $\mm \in \mSpec(R)$.
\end{proof}
\begin{definition}
    Let $C(\mm)$ be a cosilting object of cofinite type in $\Der(R_\mm)$ for each $\mm \in \mSpec(R)$ corresponding to a non-degenerate Thomason filtration $\mathbb{X}(\mm)$ in $\Spec(R_\mm)$. The family $\{C(\mm) \mid \mm \in \mSpec(R)\}$ is said to be \emph{compatible} if the family $\{\mathbb{X}(\mm) \mid \mm \in \mSpec(R)\}$ is compatible. We say that two compatible families $\{C(\mm) \mid \mm \in \mSpec(R)\}$ and $\{D(\mm) \mid \mm \in \mSpec(R)\}$ of cosilting objects of cofinite type are \emph{equivalent} if the cosilting objects $C(\mm)$ and $D(\mm)$ are equivalent for each $\mm \in \mSpec(R)$.
\end{definition}
\begin{theorem}\label{thm4.17}
    There is a bijection
    $$\left \{ \begin{tabular}{ccc} \text{ Cosilting objects $C$ }  \\ \text{ in $\Der(R)$ of cofinite type } \\ \text{up to equivalence} \end{tabular}\right \}  \xleftrightarrow{1-1}  \left \{ \begin{tabular}{ccc} \text{ Compatible families $\{C(\mm) \mid \mm \in \mSpec(R)\}$ } \\ \text{ of cosilting objects of cofinite type } \\ \text{up to equivalence} \end{tabular}\right \}$$
    induced by the assignment
    $$C \mapsto \{C^\mm \mid \mm \in \mSpec(R)\}$$
    and
    $$\{C(\mm) \mid \mm \in \mSpec(R)\} \mapsto \prod_{\mm \in \mSpec(R)}C(\mm).$$
\end{theorem}
\begin{proof}
    The assignment $C \mapsto \{C^\mm \mid \mm \in \mSpec(R)\}$ clearly preserves the appropriate equivalence classes and so is well-defined. Since $C$ is equivalent to the cosilting object $\prod_{\mm} C^\mm$ in $\Der(R)$ by Lemma~\ref{lem_cos_prod}, the assignment is injective on equivalence classes. Let $\{C(\mm) \mid \mm \in \mSpec(R)\}$ be a compatible family of cosilting objects of cofinite type which by the definition corresponds to a compatible family $\{\mathbb{X}(\mm) \mid \mSpec(R)\}$ of non-degenerate Thomason filtrations. Let $\mathbb{X}$ be the corresponding non-degenerate Thomason filtration on $\Spec(R)$ via Lemma~\ref{lem_corr_thom_nd}, which further corresponds to a cosilting object $C$ in $\Der(R)$ via Proposition~\ref{prop_nondegen_cosilting}. Let $(\Ucal,\Vcal)$ be the t-structure induced by $C$, then $C^\mm$ induces the t-structure $(\Ucal_\mm,\Vcal_\mm)$ by Lemma~\ref{lem_cosilt_coloc} and Lemma~\ref{lemma_tstr_loc}. By Proposition~\ref{prop_cg_loc}(ii) we see that $(\Ucal_\mm,\Vcal_\mm)$ corresponds to the Thomason filtration $\mathbb{X}(\mm)$ via Proposition~\ref{prop_corr_thom}. Then $(\Ucal_\mm,\Vcal_\mm)$ is the t-structure induced by $C(\mm)$ and therefore the cosilting objects $C^\mm$ and $C(\mm)$ are equivalent in $\Der(R_\mm)$ for each $\mm \in \mSpec(R)$. Finally, let us use this to show that $D = \prod_{\mm \in \mSpec(R)}C(\mm)$ is a cosilting object in $\Der(R)$ which is equivalent to $C$. We have $\Prod(C^\mm) = \Prod(C(\mm))$ for any $\mm \in \mSpec(R)$. By Lemma~\ref{lem_cos_prod}, $C$ is equivalent to $\prod_{\mm \in \mSpec(R)}C^\mm$, and so $\Prod(C) = \Prod(\prod_{\mm \in \mSpec(R)}C^\mm)$. Together, we showed that $\Prod(C) = \Prod(\prod_{\mm \in \mSpec(R)}C^\mm) = \Prod(\prod_{\mm \in \mSpec(R)}C(\mm)) = \Prod(D)$. From this, it follows easily that for any $X \in \Der(R)$ and $i \in \Z$ we have $\Hom_{\Der(R)}(X,C[i]) = 0$ if and only if $\Hom_{\Der(R)}(X,D[i]) = 0$, and so $({}^{\perp_{\leq 0}}C,{}^{\perp_{> 0}}C) = ({}^{\perp_{\leq 0}}D,{}^{\perp_{> 0}}D)$, establishing that $D$ is a cosilting object equivalent to $C$.
\end{proof}
We also have this auxiliary result using the local-global criterion for compact generation from Section~\ref{sec_three}.
\begin{proposition}
    Let $C$ be a pure-injective cosilting object in $\Der(R)$. Then $C$ is of cofinite type if and only if $C^\mathfrak{m}$ is of cofinite type for each $\mm \in \mSpec(R)$.
\end{proposition}
\begin{proof}
    Follows directly from Proposition~\ref{prop_cg_loc}.
\end{proof}
We say that a cosilting object $C$ is \emph{$n$-term} for some $n \geq 0$ if $C$ is isomorphic in $\Der(R)$ to a complex of injective $R$-modules concentrated in degrees $0,1,\ldots,n-1$.
\begin{corollary}
    Let $C$ be a cosilting object in $\Der(R)$ and $n \geq 0$. Then $C$ is $n$-term if and only if $C^\mm$ is $n$-term for any $\mm \in \mSpec(R)$. In particular, the bijection of Theorem~\ref{thm4.17} restricts for any $n \geq 0$ to a bijection
    $$\left \{ \begin{tabular}{ccc} \text{ $n$-term cosilting objects $C$ }  \\ \text{ in $\Der(R)$ of cofinite type } \\ \text{up to equivalence} \end{tabular}\right \}  \xleftrightarrow{1-1}  \left \{ \begin{tabular}{ccc} \text{ Compatible families $\{C(\mm) \mid \mm \in \mSpec(R)\}$ } \\ \text{ of $n$-term cosilting objects of cofinite type } \\ \text{up to equivalence} \end{tabular}\right \}.$$
\end{corollary}
\begin{proof}
    Let us assume that $C$ is already a complex of injective $R$-modules concentrated in degrees $0,1,\ldots,n-1$. In particular, $C^\mm = \RHom_R(R_\mm,C) \cong \Hom_R(R_\mm,C)$. As $\Hom_R(R_\mm,-)$ sends injective $R$-modules to injective $R_\mm$-modules, this establishes that $C^\mm$ is $n$-term.

    For the converse, recall from Lemma~\ref{lem_cos_prod} that $C$ is equivalent to $\prod_{\mm \in \mSpec(R)}C^\mm$. Since any injective $R_\mm$-module is also injective as an $R$-module, we conclude that $C$ is $n$-term provided that $C^\mm$ is $n$-term for any $\mm \in \mSpec(R)$.
\end{proof}
\begin{corollary}\label{cor_result_noeth}
    If $R$ is noetherian, we have bijections
    $$\left \{ \begin{tabular}{ccc} \text{ Pure-injective cosilting objects }  \\ \text{ $C$ in $\Der(R)$ up to equivalence} \end{tabular}\right \}  \xleftrightarrow{1-1}  \left \{ \begin{tabular}{ccc} \text{ Compatible families $\{C(\mm) \mid \mm \in \mSpec(R)\}$ } \\ \text{ of pure-injective cosilting objects } \\ \text{up to equivalence} \end{tabular}\right \}$$
        and
    $$\left \{ \begin{tabular}{ccc} \text{ $n$-term cosilting objects $C$} \\ \text{in $\Der(R)$ up to equivalence} \end{tabular}\right \}  \xleftrightarrow{1-1}  \left \{ \begin{tabular}{ccc} \text{ Compatible families $\{C(\mm) \mid \mm \in \mSpec(R)\}$ } \\ \text{ of $n$-term cosilting objects } \\ \text{up to equivalence} \end{tabular}\right \}.$$
\end{corollary}
\begin{proof}
    Follows immediately by recalling Theorem~\ref{noeth_coftype} and Theorem~\ref{mv_bounded}.
\end{proof}
Recall that an $R$-module $C$ is $n$-cotilting if and only if it is an $n$-term cosilting complex when considered as an object $\Der(R)$ by taking its stalk complex in degree zero. Then our correspondence also restricts to the one of \cite{tr14}. Recall from Remark~\ref{rem_noeth} that our notions of equivalence and compatible condition on families restrict perfectly well to the notions used in \cite{tr14}.
\begin{corollary}\label{cor:cotilting-correspond}
    If $R$ is noetherian, we have bijections for any $n > 0$
    $$\left \{ \begin{tabular}{ccc} \text{ $n$-cotilting modules $C$ in} \\ \text{$\ModR$ up to equivalence} \end{tabular}\right \}  \xleftrightarrow{1-1}  \left \{ \begin{tabular}{ccc} \text{ Compatible families $\{C(\mm) \mid \mm \in \mSpec(R)\}$ } \\ \text{ of $n$-cotilting modules up to equivalence} \end{tabular}\right \}.$$
\end{corollary}
\begin{proof}
    The only non-trivial task is to show that if $C$ is an $n$-cotilting module then $\mathbf{R}\Hom(R_\mm,C)$ is isomorphic to the $R$-module $\Hom_R(R_\mm,C)$ in $\Der(R)$. For this, it is sufficient to show that $\Ext_R^i(R_\mm,C) = 0$ for all $i>0$. But this follows from the well-known fact that any cotilting class in $\ModR$ contains all flat $R$-modules.
\end{proof}
\subsection{Cosilting modules over noetherian rings}
Given an $R$-module $M$, we use $\Cogen(M)$ to denote the subcategory of all $R$-modules \emph{cogenerated} by $M$, that is, all $R$-modules admitting a monomorphism into an arbitrary direct product of copies of $M$. Given a map $Q_{0}\stackrel{\eta}\rightarrow Q_{1}$ between injective $R$-modules we define a subcategory
$$\Bcal_\eta = \{M \in \ModR \mid \Hom_R(M,\eta) \text{ is an epimorphism}\}.$$
We say that an $R$-module $C$ is a \emph{cosilting module} if there is an injective copresentation $0\rightarrow C\rightarrow Q_{0}\stackrel{\eta}\rightarrow Q_{1}$ such that $\Bcal_\eta = \Cogen(C)$. We say that two cosilting modules $C,C'$ are \emph{equivalent} if they induce the same cosilting class, that is, $\Cogen(C) = \Cogen(C')$. It is well-known that cosilting modules $C$ and $C'$ are equivalent if and only if $\Prod(C) = \Prod(C')$.

The cosilting modules were introduced a module-theoretic shadows of 2-term cosilting complexes in \cite{br17}, dualizing results of \cite{li16}. More precisely, by an argument dual to that of \cite[Theorem 4.11]{li16}, we have the following result.
\begin{theorem}\label{th1}
    Let $R$ be a ring. Then there is a bijection

    $$\left \{ \begin{tabular}{ccc}\text{ $2$-term cosilting complexes } \\ \text{ up to equivalence} \end{tabular}\right \}  \xleftrightarrow{1-1}  \left \{ \begin{tabular}{ccc} \text{ cosilting $R$\text{-}modules} \\ \text{up to equivalence} \end{tabular}\right \}.$$
 This correspondence assigns to a $2$-term cosilting complex $\sigma$ the $R$-module ${\rm H}^{0}(\sigma)$.

\end{theorem}
Now let $R$ be a commutative noetherian ring. Then the equivalence classes of cosilting modules are in natural bijection with Thomason sets \cite{hr17}. Let $0\rightarrow C\rightarrow Q_{0}\stackrel{\eta}\rightarrow Q_{1}$ be an injective copresentation for an $R$-module $C$. If $C$ is cosilting  with respect to $\eta$ then it will induce a unique torsion pair $(^{\perp}C, \mathrm{Cogen}(C))$ up to equivalence, see \cite[Corollary 3.5]{br17}, where $^{\perp}C = \{M \in \ModR \mid \Hom_R(M,C) = 0\}$. We know that over a commutative ring $R$, there is a bijective correspondence between
 hereditary torsion pairs of finite type, and  Thomason subsets of $\Spec(R)$, see Theorem~\ref{thm_torsion_pair}. Note that the cosilting class $\mathrm{Cogen}(C)$ is a definable class in $\ModR$, so it is closed under direct limits.
Therefore, since $R$ is noetherian, $(^{\perp}C, \mathrm{Cogen}(C))$ is a hereditary torsion pair of finite type, see \cite[Lemma 4.2]{hr17}. The Thomason subset corresponding to $(^{\perp}C, \mathrm{Cogen}(C))$ is $Y=\bigcup\{V(I)\mid I~~f.g.~~ \text{ideal~~such~~that}~~R/I\in{^{\perp}C}\}$. 

In Proposition~\ref{prop3}, we show that this Thomason set coincides with the one obtained from the cosilting module by passing to the 2-term cosilting complex with Theorem~\ref{th1}, and then extracting the zero term Thomason set of the filtration associated via Proposition~\ref{prop_nondegen_cosilting}.

\begin{definition}
    Let $C(\mm)$ be a cosilting $R_\mm$-module for each $\mm \in \mSpec(R)$ corresponding to a  Thomason subset $X(\mm)$ in $\Spec(R_\mm)$. The family $\{C(\mm) \mid \mm \in \mSpec(R)\}$ is said to be \emph{compatible} if the family $\{X(\mm) \mid \mm \in \mSpec(R)\}$ is compatible. We say that two compatible families $\{C(\mm) \mid \mm \in \mSpec(R)\}$ and $\{D(\mm) \mid \mm \in \mSpec(R)\}$ of cosilting objects of cofinite type are \emph{equivalent} if the cosilting modules $C(\mm)$ and $D(\mm)$ are equivalent for each $\mm \in \mSpec(R)$.
\end{definition}

\begin{proposition}\label{prop3}Let $R$ be a commutative noetherian ring. Then there is a bijection

    $$\left \{ \begin{tabular}{ccc}\text{ Compatible families }\\ \text{$\{\sigma(\mm) \mid \mm \in \mSpec(R)\}$ } \\ \text{ of $2$-term cosilting complexes } \\ \text{up to equivalence} \end{tabular}\right \}  \xleftrightarrow{1-1}  \left \{ \begin{tabular}{ccc} \text{ Compatible families}\\ \text{ $\{C(\mm) \mid \mm \in \mSpec(R)\}$ } \\ \text{ of  cosilting  $R_{\mm}$\text{-}modules } \\ \text{up to equivalence} \end{tabular}\right \}.$$
 This correspondence assigns to a $2$-term cosilting complex $\sigma(\mm)$ the $R$-module ${\rm H}^{0}(\sigma(\mm))$.
\end{proposition}
\begin{proof} The assignment $\sigma(\mm)\rightarrow{\rm H}^{0}(\sigma(\mm))$ clearly preserves the equivalence and so is well-defined. By Theorem \ref{th1},  it remains to prove that the class
$\{\sigma(\mm) \mid \mm \in \mSpec(R)\}$ is compatible if and only if $\{C(\mm) \mid \mm \in \mSpec(R)\}$ is compatible.
  Let $\mathbb{X}(\mm)  = (X(\mm)_n \mid n \in \Z) $ be the non-degenerate Thomason filtration of $\Spec(R_\mm)$ corresponding to $\sigma(\mm)$  for each $\mm \in \mSpec(R)$. Recall that the family $\{\mathbb{X}(\mm) \mid \mm \in \mSpec(R)\}$ is compatible if $\{X(\mm)_n \mid \mm \in \mSpec(R)\}$ is a compatible family of Thomason subsets for each $n \in \Z$.
Let $Y(\mm)$ be the Thomason subset of $\Spec(R_\mm)$ corresponding to $C(\mm)$ for each $\mm \in \mSpec(R)$.

First, we claim that $X(\mm)_0=Y(\mm)$. In fact, by \cite[Theorem 5.1]{hr20}, we have
 \begin{center}
 $X(\mm)_0=\bigcup\{V(I_{\mm})\mid I~~ \text{is~~an~~ ideal~~of}~~R~~ \text{such~~that}~~R_\mm/I_{\mm}\in{^{\perp_{\leq0}}\sigma(\mm)}\}$,
$Y(\mm)=\bigcup\{V(J_{\mm})\mid J~~ \text{is~~ an ideal~~of}~~R~~ \text{such~~that}~~R_\mm/J_{\mm}\in{^{\perp}{\rm H}^{0}(\sigma(\mm))}\}$.
  \end{center}
Since $\sigma(\mm)$ is a $2$-term cosilting complex concentrated in degree 0 and 1, we see that  ${\rm H}^{i}(\sigma(\mm))=0$ for all $i<0$. Then the claim follows by $\Hom_{\Der(R_\mm)}(R_\mm/I_\mm,\sigma(\mm)[<0])=0$ and the isomorphism
\begin{align*}
\Hom_{\Der(R_\mm)}(R_\mm/I_\mm,\sigma(\mm)) \cong \Hom_{R_\mm}(R_\mm/I_\mm,{\rm H}^{0}(\sigma(\mm))).
 %\Hom_{\Kh(R_\mm)}(R_\mm/I_\mm,\sigma(\mm)[\leq0])
 %={\rm H}^{\leq0}\Hom_{R_\mm}(R_\mm/I_\mm,\sigma(\mm))\\
 %=&\Hom_{R_\mm}(R_\mm/I_\mm,{\rm H}^{0}(\sigma(\mm))).
\end{align*}
Moreover, one can check that $X(\mm)_n=\bigcup\{V(I_\mm)\mid I~~ \text{is~~an~~ideal~~of~~ $R$~~such~~that}~~R_\mm/I_\mm[-n]\in{^{\perp_{\leq0}}\sigma(\mm)}\}=\Spec(R_\mm)$ for all $n<0$
and $X(\mm)_n=X(\mm)_{n+1}=\cdots$ for all $n\geq1$. On the other hand, since the Thomason filtration $\mathbb{X}(\mm)$ is non-degenerate, that is, $\bigcap_{n \in \Z}X(\mm)_n = \emptyset$ and $\bigcup_{n \in \Z}X(\mm)_n = \Spec(R)$, we see that $X(\mm)_n=\emptyset$ for all $n\geq1$, and $X(\mm)_n=\Spec(R)$ for all $n<0$. Therefore, we easily get that $\{\mathbb{X}(\mm) \mid \mm \in \mSpec(R)\}$ is compatible
if and only if $\{Y(\mm) \mid \mm \in \mSpec(R)\}$ is compatible.
\end{proof}

\begin{corollary}\label{corollary:cosilting-correspond}Let $R$ be a commutative noetherian ring. Then there is a bijection

    $$\left \{ \begin{tabular}{ccc} \text{ cosilting $R$\text{-}modules} \\ \text{up to equivalence} \end{tabular}\right \}  \xleftrightarrow{1-1}  \left \{ \begin{tabular}{ccc} \text{ Compatible families $\{C(\mm) \mid \mm \in \mSpec(R)\}$ } \\ \text{ of  cosilting  $R_{\mm}$\text{-}modules up to equivalence} \end{tabular}\right \}$$
     induced by the assignment
    $$C \mapsto \{C^\mm \mid \mm \in \mSpec(R)\}$$
    and
    $$\{C(\mm) \mid \mm \in \mSpec(R)\} \mapsto \prod_{\mm \in \mSpec(R)}C(\mm).$$
\end{corollary}
\begin{proof}
By Corollary~\ref{cor_result_noeth}, we know that there  are bijections between $2$-term cosilting complexes $C$ in $\Der(R)$ up to equivalence and compatible families $\{C(\mm) \mid \mm \in \mSpec(R)\}$  of $2$-term cosilting complexes up to equivalence. Thus the result follows by Theorem \ref{th1} and Proposition \ref{prop3}.
\end{proof}
\section{Silting objects}\label{S:silting}
Let us fix an injective cogenerator $W$ in $\ModR$, and denote by $(-)^+$ the duality functor $(-)^+ = \RHom_R(-,W): \Der(R) \rightarrow \Der(R)$. The following result extends the well-known explicit duality between $n$-tilting $R$-modules and $n$-cotilting $R$-modules of cofinite type.
\begin{theorem}\cite[Theorem 3.3, Theorem 3.8]{lihrbek}\label{thm_duality}
    Let us consider the assignment $\Phi: T \mapsto T^+$ on objects of $\Der(R)$.
    Then:
    \begin{enumerate}
        \item[(i)] $\Phi$ induces an injective map from the set of equivalence classes of silting objects in $\Der(R)$ of finite type to cosilting objects in $\Der(R)$ of cofinite type.
        \item[(ii)] $\Phi$ induces a bijective map from the set of equivalence classes of bounded silting complexes in $\Der(R)$ to bounded cosilting complexes in $\Der(R)$ of cofinite type.
        \item[(iii)] If $R$ is commutative noetherian then $\Phi$ induces an bijective map from the set of equivalence classes of silting objects in $\Der(R)$ of finite type to pure-injective cosilting objects in $\Der(R)$.
    \end{enumerate}
\end{theorem}
\begin{definition}
    Let $T(\mm)$ be a silting object of finite type in $\Der(R_\mm)$ for each $\mm \in \mSpec(R)$, $C(\mm) = T(\mm)^+$ the cosilting object of cofinite type obtained via Theorem~\ref{thm_duality}, which further corresponds to a non-degenerate Thomason filtration $\mathbb{X}(\mm)$ in $\Spec(R_\mm)$ via Proposition~\ref{prop_nondegen_cosilting}. The family $\{T(\mm) \mid \mm \in \mSpec(R)\}$ is said to be \emph{compatible} if the family $\{\mathbb{X}(\mm) \mid \mm \in \mSpec(R)\}$ is compatible. We say that two compatible families $\{T(\mm) \mid \mm \in \mSpec(R)\}$ and $\{U(\mm) \mid \mm \in \mSpec(R)\}$ of silting objects of finite type are \emph{equivalent} if the silting objects $T(\mm)$ and $U(\mm)$ are equivalent for each $\mm \in \mSpec(R)$.
\end{definition}
\begin{lemma}\label{lem_loc_silt}
    Let $T \in \Der(R)$ be a silting object and $\pp \in \Spec(R)$. Then $T_\pp$ is a silting object in $\Der(R_\pp)$. If $T$ is of finite type in $\Der(R)$ then so is $T_\pp$ in $\Der(R_\pp)$.
\end{lemma}
\begin{proof}
    The proof is completely dual to that of Lemma~\ref{lem_cosilt_coloc}. The second claim is straightforward to check.
\end{proof}

A silting object $T$ in $\Der(R)$ is \emph{$n$-term} for some $n \geq 0$ if $T$ is isomorphic in $\Der(R)$ to a complex of projective $R$-modules concentrated in degrees $-(n-1),-(n-2),\ldots,-1,0$. In view of Theorem~\ref{thm_duality}, if $T$ is an $n$-term silting object then $T^+$ is clearly an $n$-term cosilting object of $\Der(R)$.

For any maximal ideal $\mm$, let us denote $(-)^{+_\mm} = \RHom_{R_\mm}(-,W^\mm)$ and note that $W^\mm$ is an injective cogenerator for the category $\ModRm$. By the adjunction, we have for any $X \in \Der(R)$ the simple formula $(X^+)^\mm = (X_\mm)^{+_\mm}$.
\begin{theorem}\label{thm_silting}
    For any $n \geq 0$, there is a bijection
    $$\left \{ \begin{tabular}{ccc} \text{ $n$-term silting objects $T$ }  \\ \text{ in $\Der(R)$ up to equivalence} \end{tabular}\right \}  \xleftrightarrow{1-1}  \left \{ \begin{tabular}{ccc} \text{ Compatible families $\{T(\mm) \mid \mm \in \mSpec(R)\}$ } \\ \text{ of $n$-term silting objects up to equivalence} \end{tabular}\right \}.$$
    If $R$ is commutative noetherian, there is also a bijection
    $$\left \{ \begin{tabular}{ccc} \text{ Silting objects $T$ }  \\ \text{ in $\Der(R)$ of finite type } \\ \text{up to equivalence} \end{tabular}\right \}  \xleftrightarrow{1-1}  \left \{ \begin{tabular}{ccc} \text{ Compatible families $\{T(\mm) \mid \mm \in \mSpec(R)\}$ } \\ \text{ of silting objects of finite type } \\ \text{up to equivalence} \end{tabular}\right \}.$$
    Both bijections are induced by the assignment $T \mapsto \{T_\mm \mid \mm \in \mSpec(R)\}$.
\end{theorem}
\begin{proof}
    Let $T$ be a silting object in $\Der(R)$, $\mm \in \mSpec(R)$, $C = T^+$ and let $\{C(\mm) \mid \mm \in \mSpec(R)\}$ be the compatible family of cosilting objects corresponding to $C$ via Proposition~\ref{prop_nondegen_cosilting}. Note that we can choose $C(\mm) = C^\mm$ for all $\mm \in \mSpec(R)$. By the formula above, we see that $C(\mm) = (T^+)^\mm = (T_\mm)^{+_\mm}$. Together with Lemma~\ref{lem_loc_silt}, this yields that $\{T_\mm \mid \mSpec(R)\}$ is a compatible family of silting objects and so both the assignment is well-defined. By Theorem~\ref{thm_duality}, the assignment $(-)^+$ is injective on equivalence classes on silting objects, and therefore so is the assignment $T \mapsto \{T_\mm \mid \mm \in \mSpec(R)\}$.

    It remains to show that the assignment $T \mapsto \{T_\mm \mid \mm \in \mSpec(R)\}$ is surjective. Let $\{T_\mm \mid \mm \in \mSpec(R)\}$ be a compatible family of silting objects of finite type and let $C(\mm) = T(\mm)^+ = T(\mm)^{+_\mm}$ for each $\mm \in \mSpec(R)$. Then $\{C(\mm) \mid \mm \in \mSpec(R)\}$ is a compatible family of cosilting objects of cofinite type, so there is a corresponding cosilting object $C \in \Der(R)$ via Theorem~\ref{thm4.17}. If either $C$ is $n$-term, or $R$ is commutative noetherian then Theorem~\ref{thm_duality} yields a silting object $T$ of finite type such that $T^+$ is equivalent to $C$ as a cosilting object. Now for reach $\mm \in \mSpec(R)$, $(T_\mm)^{+_\mm} = (T^+)^{\mm}$ is a cosilting object in $\Der(R_\mm)$ equivalent to $C(\mm)$. Therefore, Theorem~\ref{thm_duality} again implies that $T_\mm$ is equivalent to $T(\mm)$ as silting object for each $\mm \in \mSpec(R)$.
\end{proof}
\begin{remark}\label{rem_silting_coproduct}
    In view of Theorem~\ref{thm4.17}, it would be tempting to express the converse assignment of the bijection of Theorem~\ref{thm_silting} in terms of the coproduct $\bigoplus_{\mm \in \mSpec(R)}T(\mm)$. However, given a silting object $T$ the coproduct $\bigoplus_{\mm \in \mSpec(R)}T_\mm$ is often not a silting object anymore, see \cite[Remark 2.9]{tr14}. It is instructive to see where the proof of Lemma~\ref{lem_cos_prod} breaks in the dual setting --- the silting aisle $T^{\perp_{>0}}$ is usually not closed under taking colocalization. This is in contrast with cosilting coaisles, which are closed both under localization and colocalization, which is used in an essential way in obtaining results of Section~\ref{S:cosilting}.
\end{remark}
%\renewcommand\refname{\bf References}
%\begin{thebibliography}{99}
\bibliographystyle{abbrv}
\bibliography{bibitems}

\end{document}